\documentclass[a4paper,final]{amsart}

\usepackage{showkeys}
\usepackage{fullpage}
\usepackage{mathtools, amssymb, amsthm, amsfonts, mathrsfs}
\usepackage{enumitem}

\usepackage{tikz}
\usetikzlibrary{intersections, calc, arrows, cd}
\usepackage{xcolor, graphicx, here}
\definecolor{darkgreen}{rgb}{0.0, 0.6, 0.0}

\usepackage{thmtools}
\usepackage[colorlinks,citecolor=darkgreen,linkcolor=red]{hyperref}
\usepackage{cleveref}

\numberwithin{equation}{section}
\numberwithin{figure}{section}
\allowdisplaybreaks

\setenumerate{label={\upshape (\arabic*)}, nosep}
\setitemize{nosep}
\setdescription{nosep}
\newlist{enua}{enumerate}{1}
\setlist*[enua]{label={\upshape (\arabic*)}, nosep}
\newlist{enur}{enumerate}{1}
\setlist*[enur]{label={\upshape (\roman*)}, nosep}

\theoremstyle{plain}
\newtheorem{thm}{Theorem}[section]
\newtheorem{prp}[thm]{Proposition}
\newtheorem{lem}[thm]{Lemma}
\newtheorem{cor}[thm]{Corollary}
\newtheorem{fct}[thm]{Fact}

\newtheorem{thma}{Theorem}

\theoremstyle{definition}
\newtheorem{dfn}[thm]{Definition}

\newtheorem{rmk}[thm]{Remark}
\newtheorem{ex}[thm]{Example}
\newtheorem{claim}{Claim}

\crefname{thm}{Theorem}{Theorems}
\crefname{thma}{Theorem}{Theorems}
\crefname{prp}{Proposition}{Propositions}
\crefname{lem}{Lemma}{Lemmas}
\crefname{cor}{Corollary}{Corollaries}
\crefname{dfn}{Definition}{Definitions}
\crefname{rmk}{Remark}{Remarks}
\crefname{fct}{Fact}{Facts}
\crefname{ex}{Example}{Examples}
\crefname{ques}{Question}{Questions}

\newcommand{\ass}{\mathrm{ass}}

\DeclareMathOperator{\add}{\mathsf{add}} 
\newcommand{\up}{\mathsf{up}}
\newcommand{\cov}{\mathsf{cov}}
\DeclareMathOperator{\dom}{dom}

\DeclareMathOperator{\soc}{soc}

\DeclareMathOperator{\Se}{\mathsf{S}}
\DeclareMathOperator{\Bass}{\mathsf{Bass}}
\DeclareMathOperator{\Fct}{\mathsf{Fct}}

\newcommand{\F}{\mathsf{F}}

\DeclareMathOperator{\ke}{\mathsf{ke}}

\DeclareMathOperator{\torf}{\mathsf{torf}}

\DeclareMathOperator{\serre}{\mathsf{serre}}

\newcommand{\imply}{\Rightarrow}
\newcommand{\equi}{\Leftrightarrow}
\newcommand{\Imply}{\quad \Longrightarrow \quad}
\newcommand{\Equi}{\quad \Longleftrightarrow \quad}

\newcommand{\xr}[1]{\xrightarrow{\, #1 \, }}
\newcommand{\inj}{\hookrightarrow}
\newcommand{\surj}{\twoheadrightarrow}

\newcommand{\isoto}{\xr{\simeq}}

\newcommand{\iso}{\cong}

\newcommand{\bbF}{\mathbb{F}}

\newcommand{\bbN}{\mathbb{N}}

\newcommand{\bbk}{\Bbbk}

\newcommand{\catA}{\mathcal{A}}

\newcommand{\catC}{\mathcal{C}}

\newcommand{\catF}{\mathcal{F}}

\newcommand{\catS}{\mathcal{S}}

\newcommand{\catX}{\mathcal{X}}
\newcommand{\catY}{\mathcal{Y}}

\DeclareMathOperator{\Mod}{\mathsf{Mod}}
\DeclareMathOperator{\catmod}{\mathsf{mod}}
\DeclareMathOperator{\proj}{\mathsf{proj}}
\DeclareMathOperator{\fl}{\mathsf{fl}}

\DeclareMathOperator{\tf}{\mathsf{tf}}
\DeclareMathOperator{\cm}{\mathsf{cm}}

\DeclareMathOperator{\Fun}{Fun}
\newcommand{\id}{\mathrm{id}}
\DeclareMathOperator{\Hom}{Hom}
\DeclareMathOperator{\End}{End}
\DeclareMathOperator{\Ext}{Ext}

\DeclareMathOperator{\Ima}{Im}
\DeclareMathOperator{\Ker}{Ker}
\DeclareMathOperator{\Cok}{Cok}
\DeclareMathOperator{\ZZ}{Z} 

\newcommand{\pb}{\arrow[rd,"{\mathrm{PB}}",phantom]}

\DeclareMathOperator{\Spec}{Spec}
\DeclareMathOperator{\Max}{Max}
\DeclareMathOperator{\Min}{Min}
\DeclareMathOperator{\Ass}{Ass}
\DeclareMathOperator{\Assh}{Assh} 
\DeclareMathOperator{\NF}{NF} 
\DeclareMathOperator{\NZ}{NZ}

\DeclareMathOperator{\htt}{ht}
\DeclareMathOperator{\depth}{depth}
\DeclareMathOperator{\grade}{grade} 

\newcommand{\mm}{\mathfrak{m}}
\newcommand{\nn}{\mathfrak{n}}
\newcommand{\pp}{\mathfrak{p}}
\newcommand{\qq}{\mathfrak{q}}

\newcommand{\kk}{\kappa}

\DeclareMathOperator{\Supp}{Supp}

\newcommand{\LL}{\Lambda}

\newcommand{\arr}[1]{\arrow[{#1}]}
\newenvironment{bsmatrix}{\left[\begin{smallmatrix}}{\end{smallmatrix}\right]}

\begin{document}
\title{Classifying KE-closed subcategories over a commutative noetherian ring}

\author{Toshinori Kobayashi}
\address{School of Science and Technology, Meiji University, 1-1-1 Higashi-Mita, Tama-ku, Kawasaki-shi, Kanagawa 214-8571, Japan}
\email{tkobayashi@meiji.ac.jp}

\author{Shunya Saito}
\address{Graduate School of Mathematical Sciences, The University of Tokyo, 3-8-1 Komaba, Meguro-ku, Tokyo 153-8914, Japan}
\email{shunya-saito@g.ecc.u-tokyo.ac.jp}


\begin{abstract}
Let $\catmod R$ denote the category of finitely generated $R$-modules for a commutative noetherian ring $R$.
In this paper, we investigate KE-closed subcategories of $\catmod R$ as a continuation of our previous work.
We associate a function on $\Spec R$ with each KE-closed subcategory of $\catmod R$, and show that this function completely determines the original subcategory.
To classify the functions obtained from KE-closed subcategories, we introduce the notion of an $n$-Bass function for each $n\ge 0$.
We obtain a bijection between the set of KE-closed subcategories and the set of $2$-Bass functions provided that $R$ is $(S_2)$-excellent in the sense of \v{C}esnavi\v{c}ius.
\end{abstract}

\maketitle
\tableofcontents

\section{Introduction}

Understanding subcategories has been one of the central themes in both ring theory and algebraic geometry.
Gabriel provided a complete classification of the Serre subcategories of the module category over a commutative noetherian ring, or more generally, the category of coherent sheaves on a noetherian scheme \cite{Gab}.
The notion of torsion(-free) classes, which are subcategories closed under taking extensions and quotient objects (subobjects) is tightly connected with wide areas of ring theory, including tilting theory and the theory of t-structures on derived categories; see \cite{AR, AIR, IK} for instance.
It should be mentioned that the torsion-free classes of the category of finitely generated modules over a commutative noetherian ring are completely classified by Takahashi \cite{Takahashi}.

To seek generalizations of these results, it is natural to focus on subcategories that are closed under basic operations such as taking kernels, images, and cokernels of morphisms.
Recently, the notion of IE-closed subcategories was introduced to denote subcategories closed under extensions and images of morphisms, and it was applied to study representations of hereditary algebras \cite{ES}.
Enomoto showed that for commutative noetherian rings, the notions of IE-closed subcategories and torsion-free classes coincide \cite{Eno}.

A similar approach can be taken to define and study other classes of subcategories.
KE-closed subcategories form such a class, as they are defined as subcategories closed under taking kernels and extensions.
The authors investigated KE-closed subcategories with a particular emphasis on comparing these subcategories with torsion-free classes \cite{KS}.

In this paper, we continue the study of KE-closed subcategories over commutative noetherian rings. 
Despite their natural definition via basic operations, a complete classification of KE-closed subcategories remains an open problem.
We aim to classify these subcategories over a commutative noetherian ring $R$.
We associate a function $f_\catX$ on the prime spectrum $\Spec R$ of $R$ with each KE-closed subcategory $\catX$ of the category $\catmod R$ of finitely generated $R$-modules, and show in \cref{prp:X=XfX} that each such subcategory can be recovered from the function.
To characterize the functions appearing as $f_\catX$, we introduce the notion of an \emph{$n$-Bass function} for each non-negative integer $n$.
An $n$-Bass function is a function on $\Spec R$ which satisfies certain comparability conditions respect to the inclusion order on $\Spec R$ (\cref{dfn:Bass fct}).
We prove in \cref{prp:bound KE} that the function $f_\catX$ for a KE-closed subcategory $\catX$ is a $2$-Bass function.
Furthermore, in \cref{cor:bij KE 2-Bass} we establish a bijection between the set $\ke(\catmod R)$ of KE-closed subcategories of $\catmod R$ and the set of $2$-Bass functions provided that $R$ is $(S_2)$-excellent in the sense of \v{C}esnavi\v{c}ius \cite{Ces}.
Examples of $(S_2)$-excellent rings include excellent rings and homomorphic images of Cohen--Macaulay rings.
Thus our result applies to such rings.
Each $2$-Bass function is naturally associated with a pair of subsets of $\Spec R$ satisfying certain conditions.
We call such a pair \emph{a $2$-Bass sequence}; see \cref{dfn:Bass sequence}.
The set of all subsets of $\Spec R$ can be canonically embedded in the set of $2$-Bass sequences.
Therefore, we interpret Gabriel's theorem as a bijection between three sets; the set $\serre(\catmod R)$ of Serre subcategories of $\catmod R$, the set of all subsets of $\Spec R$, and the set of $0$-Bass functions.
Similarly, Takahashi's theorem can be interpreted as a bijection between the set $\torf(\catmod R)$ of torsion-free subcategories of $\catmod R$, the set of all specialization-closed subsets of $\Spec R$, and the set of $1$-Bass functions.
The results can be summarized as follows.

\begin{thma}\label{thma:Main}
We have the following commutative diagram of maps:
\begin{equation}
\begin{tikzcd}
\ke(\catmod R) \ar[r,"{f_{(-)}}",yshift=2.5pt,hook] & \Bass_2(\Spec R) \ar[l,"\catX_{(-)}",yshift=-2.5pt,two heads] \ar[r,"\simeq"] & \{\text{$2$-Bass sequences $(\Phi,\Psi)$ of $\Spec R$}\} \ar[l] \\
\torf(\catmod R) \ar[u,phantom,"\subseteq"sloped] \ar[r,"\simeq"] & \Bass_1(\Spec R) \ar[u,phantom,"\subseteq"sloped] \ar[r,"\simeq"] \ar[l] & \{\text{subsets of $\Spec R$}\} \ar[l] \ar[u,phantom,"\subseteq"sloped]\\
\serre(\catmod R) \ar[u,phantom,"\subseteq"sloped] \ar[r,"\simeq"] & \Bass_0(\Spec R) \ar[u,phantom,"\subseteq"sloped] \ar[r,"\simeq"] \ar[l] & \{\text{specialization-closed subsets of $\Spec R$}\} \ar[u,phantom,"\subseteq"sloped] \ar[l]
\end{tikzcd}
\end{equation}
The map $f_{(-)}$ is an injection with a retraction $\catX_{(-)}$.
Moreover, $f_{(-)}$ is bijective provided that $R$ is $(S_2)$-excellent.
\end{thma}

As a corollary, we characterize when the equality $\ke(\catmod R)=\torf(\catmod R)$ holds (\cref{prp:KE=torf via depth}).
In particular, we demonstrate the existence of a commutative noetherian local domain $R$ of dimension two such that $\ke(\catmod R)=\torf(\catmod R)$ (\cref{ex: 2-dim KE=torf}).
On the other hand, if the ring is assumed to be $(S_2)$-excellent, the equality $\ke(\catmod R)=\torf(\catmod R)$ implies that $\dim R\le 1$ (\cref{prp:KE=torf for S_2 excellent}). 
This extends the result of \cite{KS} by relaxing the assumption from a homomorphic image of a Cohen--Macaulay ring to an $(S_2)$-excellent ring.
In the case where $R$ is a $2$-dimensional local ring, we provide a detailed description of KE-closed subcategories that are not torsion-free in terms of certain subsets of the set of minimal primes of $R$ (\cref{prp:classify KE but not torf}).

The proof of our main results is achieved through further refinement of the techniques developed in \cite{KS}.
The fundamental approach in there involves performing a ``change of rings'' by taking the center $\ZZ_R(M)$ of the endomorphism ring of an $R$-module $M$, thereby enabling an induction argument.
Therefore, we investigate the $R$-algebra $\ZZ_R(M)$ in detail.
As one of the key tools, we obtain \cref{characterize free center}, which characterizes when $\ZZ_R(M)$ is isomorphic to the base ring $R$.
Since the centers naturally appear in more general studies of modules, this would be of independent interest.

\medskip
\noindent
\textbf{Organization.}
This paper is organized as follows.
In Section \ref{s:basic}, we recall basic properties of KE-closed subcategories established in earlier works.
In Section \ref{s:localization}, we study the behavior of KE-closed subcategories with respect to localization.
In particular, we give a ``local-to-global'' principle for KE-closed subcategories (\cref{prp:local-to-global for KE}), which provides a local criterion for determining whether a module belongs to a KE-closed subcategory.
This allows us to reduce the study of KE-closed subcategories to the local case.
In Section \ref{s:center}, we study the center of the endomorphism ring of a module
and characterize when it is isomorphic to the base ring $R$ (\cref{characterize free center}).
In Section \ref{s:KE-closed}, 
we introduce the notion of \emph{higher associated primes} 
and prove that KE-closed subcategories are recovered from them (\cref{thm:KE-closed reconstruction}).
This is essentially the most difficult part of \cref{thma:Main}, for which the results from Sections \ref{s:localization} and \ref{s:center} are used as preparation. 
In Section \ref{s:Bass}, we give correspondences between subcategories and functions on the prime spectrum and prove \cref{thma:Main}.
We explain the $(S_2)$-excellence assumption of \cref{thma:Main} in Section \ref{ss:S2}.
We also provide some consequences and explicit examples of \cref{thma:Main} in Section \ref{ss:ex Bass}.
The relationship between higher associated primes and functions corresponding to subcategories is also discussed there.

\medskip
\noindent
\textbf{Conventions.}
For a category $\catC$,
we denote by $\Hom_{\catC}(M,N)$ the set of morphisms between objects $M$ and $N$ in $\catC$.
In this paper,
we suppose that all subcategories are full subcategories closed under isomorphisms.
Thus, we often identify the subcategories with the subsets of the set of isomorphism classes
of objects.
A subcategory of an additive category is said to be \emph{additive}
if it is closed under finite direct sums.
In particular, an additive subcategory always contains zero objects.

Let $R$ be a commutative ring.
We denote by $\Spec R$ the set of prime ideals of $R$
and $\Min R$ (resp.\ $\Max R$) the set of minimal prime (resp.\ maximal) ideals of $R$.
The topology on $\Spec R$ we consider is always the Zariski topology.
By $\dim R$ we mean the Krull dimension of $R$.
For any $\pp \in \Spec R$,
we denote by $\kk(\pp):=R_{\pp}/{\pp R_{\pp}}$ the residue field of $R$ at $\pp$, and by $\htt \pp$ the height of $\pp$.
For an ideal $I$ of $R$, we denote by $V(I)$ the set of prime ideals containing $I$.
A subset $\Phi$ of $\Spec R$ is \emph{specialization-closed}
if for any inclusion $\pp \subseteq \qq$ of prime ideals, $\pp \in \Phi$ implies $\qq \in \Phi$.
We denote by $\catmod R$ the category of finitely generated (left) $R$-modules.
For $M\in \catmod R$, we denote by $\Ass M$ the set of associated prime ideals of $M$, by $\Supp M$ the support of $M$, and by $\ZZ_R(M)$ the center of the endomorphism ring $\End_R(M)$ of $M$.

\medskip
\noindent
\textbf{Acknowledgement.}
The first author was partly supported by JSPS KAKENHI Grant Number 25K17240.
The second author is supported by JSPS KAKENHI Grant Number JP24KJ0057.

\section{Preliminaries}\label{s:basic}
In this section, we recall basic properties of KE-closed subcategories established in the earlier works.
Throughout the section, let $R$ denote a commutative noetherian ring.

\subsection{Depth and the associated primes}
We begin by collecting some basic facts on the associated primes and the depth of an $R$-module, which will be used throughout this paper.
\begin{fct}[{\cite[Chapter IV, \S 1.1, Proposition 4]{Bourbaki CA}}]\label{fct:Goto-Watanabe}
Let $M \in \catmod R$.
Consider a decomposition $\Ass M = \Phi \sqcup \Psi$ as a set.
Then there exists an exact sequence $0\to L \to M\to N\to 0$
such that $\Ass L=\Phi$ and $\Ass N=\Psi$.
\end{fct}

\begin{fct}[{\cite[Chapter IV, \S 2.1, Proposition 10]{Bourbaki CA}}] \label{ass hom}
$\Ass_R(\Hom_R(M,N)) = \Supp M \cap \Ass N$ for any $M,N \in \catmod R$.
\end{fct}

Assume that $(R,\mm)$ is local.
The \emph{depth} of $M\in \catmod R$ is defined as the infimum $\inf\{n \mid n\ge 0, \Ext^n(k,M)\not=0 \}$ and denoted by $\depth_R M$.
We adopt the convention $\depth_R 0=\infty$.
For an $R$-module $M\in\catmod R$, it is well known that $\pp\in \Ass M$ exactly when $\depth_{R_\pp}M_\pp=0$.

The following is well known to experts as Bass' lemma; see also \cite[Lemma 6.2]{Takahashi2}.

\begin{fct}[{\cite[(3.1) Lemma]{Bass}}] \label{Bass lemma}
Let $\pp, \qq \in \Spec R$ with $\pp \subseteq \qq$.
Let $M\in \catmod R$.
Then we have
\[
\depth_{R_\qq} M_\qq \le \depth_{R_\pp} M_\pp + \htt(\qq/\pp).
\]
\end{fct}

Observations on the depth of $R/\pp$ and $\pp$ at a localization are given here.

\begin{lem}\label{prp:deoth R/p}
Let $\pp, \qq \in \Spec R$ with $\pp \subseteq \qq$.
\begin{enua}
\item
$\pp = \qq$ if and only if $\depth_{R_{\qq}}(R/\pp)_{\qq}=0$.
\item
Suppose that $\pp \subsetneq \qq$. Then we have
\[
1 \le \depth_{R_{\qq}}(R/\pp)_{\qq} \le \htt (\qq/\pp).
\]
\end{enua}
\end{lem}
\begin{proof}
(1)
It is clear that $\pp=\qq$ implies $\depth_{R_\qq} (R/\pp)_{\qq} = 0$ 
since $(R/\pp)_{\qq}$ is a field.
Suppose that $\pp \subsetneq \qq$.
Then $(R/\pp)_{\qq}$ is an integral local domain which is not a field.
Thus, its maximal ideal contains a nonzero divisor,
which implies $\depth_{R_\qq} (R/\pp)_\qq \ge 1$.

(2)
The first and second inequalities follow from (1) and  \cref{Bass lemma}, respectively.
\end{proof}

\begin{lem}\label{prp:depth p}
For any prime ideals $\pp, \qq \in \Spec R$, we have 
\[
\depth_{R_{\qq}} \pp_{\qq}=
\begin{cases}
\depth R_{\qq} & \text{if $\qq \not\supseteq \pp$,}\\
\min\{1,\depth R_\qq\} & \text{if $\qq = \pp$ and $\pp_\qq\not=0$,}\\
\infty & \text{if $\qq = \pp$ and $\pp_\qq=0$,}\\
\ge \min\{2,\depth R_\qq\} & \text{if $\qq \supsetneq \pp$.}\\
\end{cases}
\]
\end{lem}
\begin{proof}
Assume $\qq \not\supseteq \pp$.
Then $\pp_\qq=R_\qq$.
Therefore, $\depth_{R_\qq} \pp_\qq=\depth R_\qq$.

Assume $\qq=\pp$ and $\depth R_\qq=0$.
Then the socle $\mathrm{soc}_{R_\qq}(R_\qq)$ of $R_\qq$ is nonzero.
Moreover, we have that $\mathrm{soc}_{R_\qq}(R_\qq) \subseteq \pp_\qq$ if and only if $\pp_\qq\not=0$.
Thus, it follows that $\depth \pp_\qq=0$ if $\pp_\qq\not=0$ and that $\depth \pp_\qq=\infty$ if $\pp_\qq=0$.

In the remaining cases, the assertion follows from 
the short exact sequence $0 \to \pp_\qq \to R_\qq \to (R/\pp)_\qq \to 0$, the depth lemma, and \cref{prp:deoth R/p}.
\end{proof}

\subsection{KE-closed subcategories in an abelian category}
This subsection is devoted to collecting general properties of KE-closed subcategories in an abelian category.
First of all, we give definitions of some basic classes of subcategories of an abelian category.

\begin{dfn}
Let $\catA$ be an abelian category and $\catX$ its additive subcategory.
\begin{enua}
\item
$\catX$ is said to be \emph{closed under extensions} (or \emph{extension-closed})
if for any exact sequence $0 \to A \to B \to C \to 0$,
we have that $A,C \in \catX$ implies $B\in \catX$.
\item
$\catX$ is said to be \emph{closed under subobjects} (resp.\ \emph{quotients})
if for any injection $A \inj X$ (resp.\ surjection $X \surj A$) in $\catA$ with $X\in \catX$,
we have that $A \in \catX$.
\item
$\catX$ is said to be \emph{closed under kernels}
if for any morphism $f \colon X \to Y$ in $\catA$ with $X,Y \in \catX$,
we have that $\Ker f \in \catX$.
\item
$\catX$ is called a \emph{Serre subcategory}
if it is closed under subobjects, quotients and extensions.
\item
$\catX$ is called a \emph{torsion-free class}
if it is closed under extensions and subobjects.
\item
$\catX$ is said to be \emph{KE-closed} 
if it is closed under kernels and extensions.
\item
We denote by $\serre \catA$, $\torf \catA$ and $\ke \catA$,
the set of Serre subcategories, torsion-free classes, KE-closed subcategories of $\catA$, respectively.
\end{enua}
\end{dfn}
The hierarchy of these subcategories is displayed as follows:
\[
\text{Serre} \Imply \text{torsion-free} \Imply \text{KE-closed}.
\]

%
%

Next, we recall the notion of Serre subcategories and torsion-free classes in an extension-closed subcategory.
\begin{dfn}\label{dfn:subcat in ex cat}
Let $\catX$ be an extension-closed subcategory of an abelian category $\catA$.
\begin{enua}
\item
A \emph{conflation} of $\catX$ is 
an exact sequence $0 \to A \to B \to C \to 0$ of $\catA$
such that $A$, $B$ and $C$ belong to $\catX$.
\item
An additive subcategory $\catS$ of $\catX$ is said to be \emph{closed under conflations}
if for any conflation $0\to X \to Y \to Z \to 0$ in $\catX$,
we have that $X,Z \in \catS$ implies $Y\in \catS$.
\item
An additive subcategory $\catS$ of $\catX$ is said to be \emph{closed under admissible subobjects}
if for any conflation $0\to X \to Y \to Z \to 0$ in $\catX$,
we have that $Y\in\catS$ implies $X\in\catS$.
\item
A \emph{torsion-free class of $\catX$} is an additive subcategory of $\catX$
closed under conflations and admissible subobjects.
\end{enua}
\end{dfn}

We can characterize KE-closed subcategories in terms of torsion-free classes.
For a subcategory $\catX$ of an abelian category $\catA$,
we denote by $\F(\catX)$ the smallest torsion-free class containing $\catX$.
We call it the \emph{torsion-free closure} of $\catX$.
\begin{fct}[{\cite[Proposition 3.1]{KS}}]\label{prp:KE=torf in torf}
The following are equivalent for a subcategory $\catX$ of an abelian category $\catA$.
\begin{enur}
\item
$\catX$ is a KE-closed subcategory of $\catA$.
\item
$\catX$ is a torsion-free class of $\F(\catX)$.
\item
There exists a torsion-free class $\catF$ of $\catA$
such that $\catX$ is a torsion-free class of $\catF$.
\end{enur}
\end{fct}

\subsection{KE-closed subcategories over a commutative noetherian ring}
We now focus on the fact on KE-closed subcategories in the module category $\catmod R$.
One of the most remarkable properties of KE-closed subcategories in $\catmod R$
is that they are closed under taking Hom-sets and centers.
Recall that we denote by $\ZZ_R(M):=Z(\End_R(M))$ the center of the endomorphism algebra of an $R$-module $M$.
\begin{fct}[{\cite[Lemma 2.4 (2)]{IMST}, \cite[Lemma 4.28]{KS}}]\label{fct:Hom-ideal}
Let $\catX$ be a KE-closed subcategory of $\catmod R$ and $M\in \catX$.
\begin{enua}
\item For any $X\in \catmod R$,
the $R$-module $\Hom_R(X,M)$ belongs to $\catX$.
\item The $R$-module $\ZZ_R(M)$ belongs to $\catX$.
\end{enua}
\end{fct}

Next, we recall Takahashi's classification of the torsion-free classes of $\catmod R$
since KE-closed subcategories and torsion-free classes are closely related by \cref{prp:KE=torf in torf}.
Consider the following assignments:
\begin{itemize}
\item
For a subset $\Phi$ of $\Spec R$,
define a subcategory of $\catmod R$ by
\[
\catmod^{\ass}_{\Phi} R:=\{M\in \catmod R \mid \Ass M\subseteq \Phi\}.
\]
\item
For a subcategory $\catX$ of $\catmod R$,
define a subset of $\Spec R$ by
\[
\Ass\catX :=\bigcup_{M\in\catX} \Ass M.
\]
\end{itemize}

\begin{fct}[{\cite[Theorem 4.1]{Takahashi}}]\label{fct:Takahashi}
The assignments $\catX \mapsto \Ass \catX$ and $\Phi \mapsto \catmod^{\ass}_{\Phi} R$ 
give rise to mutually inverse bijections between 
$\torf(\catmod R)$ and the power set of $\Spec R$.
\end{fct}

From this, we can describe the torsion-free closure of a subcategory.
\begin{cor}\label{prp:torf closure for comm ring}
For a subcategory $\catX$ of $\catmod R$,
we have $\F(\catX)=\catmod^\ass_{\Ass \catX} R$.\qed
\end{cor}

Torsion-free classes are KE-closed subcategories,
and thus the inclusion $\torf(\catmod R)\subseteq \ke(\catmod R)$ holds true in general.
It is natural to ask when $\torf(\catmod R) = \ke(\catmod R)$.
It holds true under the assumption $\dim R\le 1$;
see \cref{prp:KE=torf via depth,prp:KE=torf for S_2 excellent} for more general discussions on the equality $\ke(\catmod R)=\torf(\catmod R)$.
\begin{fct}[{\cite[Theorem 4.30]{KS}}] \label{KE=torf in 1-dim}
If $\dim R\le 1$, then we have $\ke(\catmod R)=\torf(\catmod R)$.
\end{fct}

We will use the following observations in Section \ref{s:KE-closed}.
\begin{cor}\label{KE desired Ass submodule}
Let $\catX$ be a KE-closed subcategory of $\catmod R$.
Then for any $M\in \catX$ and any subset $\Phi$ of $\Ass M$,
there is an exact sequence $0\to L \to M\to N\to 0$
such that $L \in \catmod_{\Phi}^{\ass} R$ and $\Ass N=\Ass M \setminus \Phi$.
\end{cor}
\begin{proof}
There is an exact sequence $0 \to L \to M \to N \to 0$
with $\Ass L = \Phi$ and $\Ass N = \Ass M \setminus \Phi$
by \cref{fct:Goto-Watanabe}.
This sequence is a conflation in $\catmod^{\ass}_{\Ass\catX} R$.
Since $\catX$ is a torsion-free class of $\catmod^{\ass}_{\Ass\catX} R$ by \cref{prp:KE=torf in torf},
we obtain that $L \in \catX$.
\end{proof}

\begin{fct}[{\cite[Lemma 4.7]{KS}}]\label{restrict KE by torf}
Let $\catX$ be a KE-closed subcategory of $\catmod R$,
and let $\Phi$ be a subset of $\Ass \catX$.
Then $\catX \cap \catmod^{\ass}_{\Phi} R$ is a KE-closed subcategory of $\catmod R$
such that $\Ass\left(\catX\cap \catmod^{\ass}_{\Phi} R\right)=\Phi$.
\end{fct}

\begin{fct}[{\cite[Lemma 4.8]{KS}}]\label{prp:ke ass max ideal}
Let $\catX$ be a KE-closed subcategory of $\catmod R$.
If $\Ass\catX=\{\mm\}$ for some $\mm \in \Max R$,
then $\catX = \catmod^{\ass}_{\{\mm\}}R$.
\end{fct}

\begin{cor}\label{max ideal KE contains fl}
Let $\catX$ be a KE-closed subcategory of $\catmod R$.
If $\mm \in \Max R \cap \Ass \catX$,
then $\catmod^{\ass}_{\{\mm\}}R \subseteq \catX$.
\end{cor}
\begin{proof}
We have
$\catmod^{\ass}_{\{\mm\}}R = \catmod^{\ass}_{\{\mm\}}R \cap \catX \subseteq \catX$
by \cref{restrict KE by torf,prp:ke ass max ideal}.
\end{proof}

Dominant resolving subcategories are another class of subcategories that closely relates to KE-closed subcategories.
\begin{dfn} \label{dfn:dom res}
Let $\catX$ be a subcategory of $\catmod R$ and $M\in \catmod R$.
\begin{enua}
\item If $0 \to N \to P_{n-1} \to \cdots \to P_1 \to P_0 \to M \to 0$ is an exact sequence in $\catmod R$ such that $P_i$ are $R$-projective for all $0\le i\le n-1$, then $N$ is called an $n$-th syzygy module of $M$.
\item $\catX$ is called \emph{resolving} if it contains all projective $R$-modules, and is closed under direct summands, extensions, and kernels of epimorphisms.
\item $\catX$ is
called \emph{dominant} if for all $M\in\catmod R$, there exists an integer $n\ge 0$ such that $\catX$ contains the $n$-th syzygy modules of $M$.
\end{enua}
\end{dfn}

The dominant resolving subcategories of $\catmod R$ were classified for a Cohen-Macaulay ring $R$ \cite[Theorem 1.4]{DT} and for a general ring $R$ with some assumption \cite[Theorem 6.9]{Takahashi2}.
The following fact gives an explicit description of dominant resolving subcategories.
\begin{fct}[{\cite[Corollary 4.6 and Proposition 5.3(1)]{Takahashi2}}] \label{thm:dom res}
Let $\catX$ be a resolving subcategory of $\catmod R$.
Then the following are equivalent.
\begin{enua}
\item $\catX$ is dominant.
\item The equality below holds true:
\[
\catX=\{M\in\catmod R\mid \depth_{R_\pp} M_\pp \ge \inf_{X\in\catX}\{\depth_{R_\pp} X_\pp\}\text{ for all }\pp\in\Spec R\}.
\]
\end{enua}
\end{fct}

The relationship between KE-closed subcategories and dominant resolving subcategories is the following.
\begin{fct}[{\cite[Corollary 4.22]{KS}}] \label{KE-closed res}
Let $\catX$ be a subcategory of $\catmod R$.
Then the following are equivalent.
\begin{enua}
\item $\catX$ is a KE-closed subcategory, and $R\in\catX$.
\item $\catX$ is a dominant resolving subcategory, and
\[
\inf_{X\in\catX}\{\depth_{R_\pp} X_\pp\} \le 2 \text{ for all }\pp\in\Spec R.
\]
\end{enua}
\end{fct}

We put here a simple criterion for the containment $R\in \catX$.

\begin{fct}[{\cite[Example 4.5]{IMST}}] \label{imst}
Let $\catX$ be an extension-closed subcategory of $\catmod R$.
If $\mm\in \catX$ for some maximal ideal $\mm$ of $R$,
then $R\in\catX$.
\end{fct}

\section{Localization and subcategories}\label{s:localization}
Let $R$ be a commutative noetherian ring and let $S$ be a multiplicatively closed subset of $R$.
For a subcategory $\catX$ of $\catmod R$, we denote by $\catX_S$ the subcategory $\{M_S\mid M\in\catX\}$ of $\catmod R_S$.
If $\pp$ is a prime ideal of $R$ and $S=R\setminus \pp$, then $\catX_S$ is denoted by $\catX_{\pp}$.

We present here two basic lemmas on splitting morphisms.

\begin{lem} \label{split_local_global}
Let $R$ be a commutative noetherian ring.
Let $f\colon M \to N$ be a homomorphism in $\catmod R$.
Then the following are equivalent.
\begin{enua}
\item $f$ is a split monomorphism (resp.\ split epimorphism).
\item $f_\pp$ is a split monomorphism (resp.\ split epimorphism) for all $\pp\in \Spec R$.
\item $f_\mm $ is a split monomorphism (resp.\ split epimorphism) for all $\mm \in \Max R$. 
\end{enua}
\end{lem}

\begin{proof}
We only need to prove (2)$\Rightarrow$(1).
Assume that $f_\pp$ is a split monomorphism for all $\pp\in \Spec R$.
Then the induced sequence
\[
\Hom_{R_\pp}(N_\pp,M_\pp) \xrightarrow[]{\Hom_{R_\pp}(f_\pp,M_\pp)} \Hom_{R_\pp}(M_\pp,M_\pp) \to 0
\]
is exact.
Since $M,N$ are finitely presented, this yields that
\[
\Hom_R(N,M)_\pp \xrightarrow[]{\Hom_R(f,M)_\pp} \Hom_R(M,M)_\pp \to 0
\]
 is exact.
It follows that $\Hom_R(f,M)$ itself is an epimorphism, and hence $f$ is a split monomorphism.

The assertion on split epimorphisms can be proved in a similar way.
\end{proof}

\begin{lem} \label{loc split mono}
Let $R$ be a commutative noetherian ring, $M$ be a finitely generated $R$-module, $\catX$ be an additive subcategory of $\catmod R$, and $\Phi \subseteq \Spec R$ be an open subset of $\Spec R$.
Assume that $M_\pp \in \catX_\pp$ for any $\pp\in \Phi$.
Then there exists $N\in \catX$ and a homomorphism $g\colon M \to N$ such that $g_\pp$ is a split monomorphism for all $\pp\in \Phi$.
\end{lem}

\begin{proof}
Take $\pp\in \Phi$.
Since $M_\pp\in \catX_\pp$, there exists $M^{\pp}\in \catX$ such that $M_\pp \cong (M^{\pp})_\pp$.
Let $f\colon M_\pp \to (M^{\pp})_\pp$ be an isomorphism.
Then there exists an $R$-homomorphism $f^{\pp}\colon M \to M^{\pp}$ and an element $s\in R\setminus \pp$ such that $s^{-1}(f^{\pp})_\pp=f$.
Consider the following open subset of $\Spec R$:
\[
U(\pp)\coloneqq \{\qq\in \Spec R \mid (f^{\pp})_\qq \text{ is an isomorphism}\}=\Spec R \setminus \left(\Supp(\Ker f^{\pp})\cup \Supp (\Cok f^{\pp})\right).
\]
Note that $\pp\in U(\pp)$.
It follows that $\{U(\pp)\}_{\pp\in \Phi}$ is an open covering of $\Phi$.
Since $\Spec R$ is noetherian, its open subset $\Phi$ is quasi-compact.
Thus there are finitely many prime ideals $\pp_1,\dots,\pp_n$ in $\Phi$ such that $\Phi \subseteq \bigcup_{i=1,\dots,n}U(\pp_i)$.
For $i=1,\dots, n$, let $f_i$ denote the $R$-homomorphism $f^{\pp_i}$.
Consider the following $R$-homomorphism:
\[
g\coloneqq {}^t[f_1,\dots,f_n]\colon M \to \bigoplus_{i=1,\dots,n}M^{\pp_i}.
\]
If $\pp\in \Phi$, then $\pp$ belongs to $U(\pp_i)$ for some $i$.
For such $i$, $(f_i)_\pp$ is an isomorphism, and thus $g_\pp={}^t[(f_1)_\pp,\dots,(f_i)_\pp,\dots,(f_n)_\pp]$ is a split monomorphism.
\end{proof}

We obtain a ``local-to-global'' principle for membership of a given subcategory.

\begin{prp} \label{loc containment}
Let $R$ be a commutative noetherian ring and $M\in \catmod R$.
Let $\catX$ be an additive subcategory of $\catmod R$ closed under direct summands.
Then the following are equivalent.
\begin{enua}
\item $M\in \catX$.
\item $M_\pp \in \catX_\pp$ for all $\pp\in \Spec R$.
\item $M_\mm \in \catX_\mm$ for all $\mm \in \Max R$. 
\end{enua}
\end{prp}

\begin{proof}
The implications (1)$\Rightarrow$(2)$\Leftrightarrow$(3) are clear.
We verify the implication (2)$\Rightarrow$(1).
Assume that the condition (2) holds.
Applying \cref{loc split mono} to $\Spec R$, we get a module $N\in \catX$ and a homomorphism $f\colon M \to N$ such that $f_\pp$ is a split monomorphism for any $\pp\in \Spec R$.
Applying \cref{split_local_global}, we see that $f$ itself is a split monomorphism.
Therefore, $M$ is a direct summand of $N$.
Since $\catX$ is closed under direct summands, $M$ belongs to $\catX$.
\end{proof}

We observe that for a KE-closed subcategory $\catX$ of $\catmod R$, the localization $\catX_S$  is KE-closed.
This is stated as follows.

\begin{prp}\label{prp:localize KE again KE}
Let $R$ be a commutative noetherian ring and let $S$ be a multiplicatively closed subset of $R$.
Let $\catX$ be a subcategory of $\catmod R$.
Assume that $\catX$ is closed under extensions (resp.\ closed under kernels, KE-closed) in $\catmod R$.
Then $\catX_S$ is closed under extensions (resp.\ closed under kernels, KE-closed) in $\catmod R_S$.
\end{prp}

\begin{proof}
We remark that every module in $\catmod R_S$ is of the form $X_S$ for some $X\in\catmod R$.

First assume that $\catX$ is closed under extensions in $\catmod R$.
Let 
$
0 \to L_S \to M_S \to N_S \to 0
$
be a short exact sequence in $\catmod R_S$.
Assume $L_S,N_S\in \catX_S$.
Thus we may assume $L,N\in \catX$.
Then the short exact sequence corresponds to an element $\sigma\in \Ext^1_{R_S}(L_S,N_S)$.
Since $R$ is noetherian and $L\in\catmod R$, $\Ext^1_{R_S}(L_S,N_S)\cong \Ext^1_R(L,N)_S$.
In particular, there exist $s\in S$ and $\tau\in \Ext^1_R(L,N)$ such that $\sigma =s^{-1}\tau$.
Let
$
0 \to L \to M' \to N \to 0
$
be the short exact sequence in $\catmod R$ corresponding to $\tau$.
Note that $M'\in \catX$ since $\catX$ is closed under extensions.
The equality $\sigma =s^{-1}\tau$ gives the following commutative diagram of $R_S$-modules
\[
\begin{tikzcd}
0 \ar[r] & L_S \ar[r] \arr{d,equal} & M_S \ar[r] \ar[d] \pb & N_S \ar[r] \ar[d, "s^{-1}"] & 0\\
0 \ar[r] & L_S \ar[r] & M'_S \ar[r] & N_S \ar[r] & 0
\end{tikzcd}
\]
with exact rows.
This shows that $M_S\cong M'_S$ since $s^{-1}$ induces an $R_S$-isomorphism on $N_S$.
In particular, $M_S\in \catX_S$.
Consequently, $\catX_S$ is closed under extensions in $\catmod R_S$.

Next assume that  $\catX$ is closed under kernels.
Let $M,N$ be $R$-modules in $\catX$ and let $f\colon M_S \to N_S$ be an $R_S$-homomorphism.
Since $M\in \catmod R$, $f$ can be written as $s^{-1}g$ for some $g\colon M \to N$ and $s\in S$.
Note that $\Ker g\in \catX$ since $\catX$ is closed under kernels.
We obtain a commutative diagram of $R_S$-modules
\[
\begin{tikzcd}
0 \ar[r] & (\Ker g)_S \ar[r] \ar[d] & M_S \ar[r, "g_S"] \ar[d, equal] & N_S \ar[d, "s^{-1}"]\\
0 \ar[r] & \Ker f \ar[r] & M_S \ar[r, "f"] & N_S
\end{tikzcd}
\]
with exact rows.
This shows that $(\Ker g)_S\cong \Ker f$ since $s^{-1}$ induces an $R_S$-isomorphism on $N_S$.
In particular, $\Ker f\in \catX_S$.
Consequently, $\catX_S$ is closed under kernels in $\catmod R_S$.
\end{proof}

We obtain a ``local-to-global'' principle for KE-closed subcategories.

\begin{cor}\label{prp:local-to-global for KE}
Let $R$ be a commutative noetherian ring and $S$ a multiplicatively closed subset of $R$.
Let $\catX$ be an additive subcategory of $\catmod R$ closed under direct summands.
Then the following are equivalent.
\begin{enua}
\item $\catX$ is closed under extensions (resp.\ closed under kernels, KE-closed) in $\catmod R$.
\item $\catX_\pp$ is closed under extensions (resp.\ closed under kernels, KE-closed) in $\catmod R_\pp$ for all $\pp\in \Spec R$.
\item $\catX_\mm$ is closed under extensions (resp.\ closed under kernels, KE-closed) in $\catmod R_\mm$ for all $\mm\in \Max R$.
\end{enua}
\end{cor}

\begin{proof}
The implications (1)$\Rightarrow$(2)$\Rightarrow$(3) follow from \cref{prp:localize KE again KE}.

(3)$\Rightarrow$(1): First assume that $\catX_\mm$ is closed under extensions for all $\mm\in\Max R$.
We show that $\catX$ is closed under extensions.
Let $0 \to L \to M \to N \to 0$ be a short exact sequence of $R$-modules such that $L,N\in \catX$.
Localizing at any maximal ideal $\mm$, we have that $0 \to L_\mm \to M_\mm \to N_\mm \to 0$ is exact.
Since $L_\mm,N_\mm\in \catX_\mm$ and $\catX_\mm$ is closed under extensions, $M_\mm$ belongs to $\catX_\mm$.
Applying \cref{loc containment}, we obtain that $M\in \catX$.

Next assume that $\catX_\mm$ is closed under kernels for all $\mm\in\Max R$.
We show that $\catX$ is closed under kernels.
Let $f\colon M \to N$ be a homomorphism such that $M,N\in \catX$.
By a similar argument as in the above, we see that $\Ker f_\mm\in \catX_\mm$ for any maximal ideal $\mm$.
Then, applying \cref{loc containment}, we obtain that $\Ker f\in \catX$.
\end{proof}

\section{The centers of endomorphism algebras}\label{s:center}
Throughout this section, we fix a commutative noetherian ring $R$.
We study the center $Z_R(M):=Z(\End_R(M))$ of the endomorphism algebra of an $R$-module $M$.
Let us explain our motivation.
Let $M$ be an object of a KE-closed subcategory $\catX$ of $\catmod R$.
Then $\ZZ_R(M) \in \catX$ by \cref{fct:Hom-ideal} (2).
If ``$\ZZ_R(M) \iso R$ holds'', then $\catX$ is a dominant resolving subcategory
and has an explicit description, see \cref{KE-closed res}.
In general, such an isomorphism does not exist even if we choose a nice object $M$ of $\catX$.
However, the analysis of the structure of $\ZZ_R(M)$ gives clues to study KE-closed subcategories as we will see in the next section.

\subsection{The centers of endomorphism algebras of direct sums}

We first see that $\ZZ_R(M)$ only depends on the additive closure $\add M$ of $M$.
Recall that $\add M$ is the subcategory of $\catmod R$ consisting of direct summands of finite coproducts of $M$.
It is the smallest additive subcategory of $\catmod R$ containing $M$ and closed under taking direct summands.
\begin{prp} \label{prp:isom_center_on_add_equiv}
If $M, N \in \catmod R$ with $\add M = \add N$,
then $\ZZ_R(M) \iso \ZZ_R(N)$ as $R$-algebras.
\end{prp}
\begin{proof}
Recall that the category $\Mod \catX$ of modules over a (small) $R$-linear category $\catX$ is defined by the category of $R$-linear functors from $\catX$ to $\Mod R$, that is,
$\Mod \catX := \Fun_R(\catX, \Mod R)$.
If $\catX$ consists of a single object whose endomorphism ring is $A$,
then $\Mod \catX$ coincides with the category $\Mod A$ of $A$-modules.
It is well known that there is an $R$-linear equivalence $\Mod(\add M) \simeq \Mod \End_R(M)$ for any $M\in \Mod R$.
For example, it easily follows from \cite[Corollary 6.11]{Bu}.
Applying this to our setting, we get $\Mod\End_R(M) \simeq \Mod(\add M) = \Mod(\add N) \simeq \Mod\End_R(N)$.
Since the centers of rings are Morita invariant,
we obtain an $R$-algebra isomorphism $\ZZ_R(M) \iso \ZZ_R(N)$ (cf.\ \cite[Proposition 21.10]{AF}).
\end{proof}

\begin{cor}
We have an $R$-algebra isomorphism $\ZZ_R(M^{\oplus n}) \iso \ZZ_R(M)$ for any $M \in \catmod R$ and $n\ge 0$.\qed
\end{cor}

Next, we study the structure of $\ZZ_R(M\oplus N)$ for $M,N \in \catmod R$.
Consider the following natural morphisms:
\[
i_M\colon M \to M\oplus N,\quad
i_N\colon N \to M\oplus N,\quad
p_M\colon M \oplus N \to M,\quad
p_N\colon M \oplus N \to N.
\]
Then we have an $R$-algebra isomorphism
\[
\End_{R}(M\oplus N) \isoto
\begin{bmatrix}
\End_{R}(M) & \Hom_{R}(N,M) \\
\Hom_{R}(M,N) & \End_{R}(N)
\end{bmatrix},\quad
f \mapsto
\begin{bmatrix}
p_M \circ f \circ i_M & p_M \circ f \circ i_N \\
p_N \circ f \circ i_M & p_N \circ f \circ i_N
\end{bmatrix}.
\]
Here, the right hand side has a ring structure 
defined by the matrix multiplication.
We often identify them by this isomorphism.
We consider $\End_{R}(M) \times \End_{R}(N)$ as a subring of $\End_{R}(M\oplus N)$
via an injective $R$-algebra homomorphism
\[
\End_{R}(M) \times \End_{R}(N) \inj \End_{R}(M\oplus N),\quad
(f,g) \mapsto 
\begin{bmatrix}
f & 0 \\
0 & g
\end{bmatrix}.
\]

\begin{lem}\label{prp:Z(M oplus N)}
We have $\ZZ_R(M\oplus N) \subseteq \ZZ_R(M) \times \ZZ_R(N)$.
\end{lem}
\begin{proof}
Take 
$\begin{bsmatrix}
f_{11} & f_{12}\\
f_{21} & f_{22}
\end{bsmatrix}
\in \ZZ_R(M\oplus N)$.
Then we have
\[
\begin{bmatrix}
f_{11} & 0\\
f_{21} & 0
\end{bmatrix}
=
\begin{bmatrix}
f_{11} & f_{12}\\
f_{21} & f_{22}
\end{bmatrix}
\begin{bmatrix}
\id_M & 0 \\
0 & 0
\end{bmatrix}
=
\begin{bmatrix}
\id_M & 0 \\
0 & 0
\end{bmatrix}
\begin{bmatrix}
f_{11} & f_{12}\\
f_{21} & f_{22}
\end{bmatrix}
=
\begin{bmatrix}
f_{11} & f_{12}\\
0 & 0
\end{bmatrix}.
\]
Thus, we have $f_{12}=0=f_{21}$.
It means that $\ZZ_R(M\oplus N) \subseteq \End_{R}(M) \times \End_{R}(N)$.
Moreover, we can conclude that
$\ZZ_R(M\oplus N) \subseteq \ZZ_R(M) \times \ZZ_R(N)$
by the following equality for any $g\in \End_R(M)$ and $h\in \End_R(N)$:
\[
\begin{bmatrix}
f_{11}\circ g & 0\\
0 & f_{22} \circ h
\end{bmatrix}
=
\begin{bmatrix}
f_{11} & 0\\
0 & f_{22}
\end{bmatrix}
\begin{bmatrix}
g & 0 \\
0 & h
\end{bmatrix}
=
\begin{bmatrix}
g & 0 \\
0 & h
\end{bmatrix}
\begin{bmatrix}
f_{11} & 0\\
0 & f_{22}
\end{bmatrix}
=
\begin{bmatrix}
g \circ f_{11} & 0\\
0 & h \circ f_{22}
\end{bmatrix}.
\]
\end{proof}

Thus, we have the natural $R$-algebra homomorphism
\[
\phi_M \colon \ZZ_R(M\oplus N) \subseteq \ZZ_R(M) \times \ZZ_R(N) \surj \ZZ_R(M).
\]
In some cases, this morphism is injective, so ``$\ZZ_R(M\oplus N)$ is smaller than $\ZZ_R(M)$''.
\begin{prp}\label{prp:Z(M oplus N) to Z(M) inj}
Let $M,N \in \catmod R$.
\begin{enua}
\item
If there is an epimorphism $M^{\oplus n} \surj N$ for some $n\ge 1$,
then $\phi_M$ is injective.
\item
If there is a monomorphism $N \inj M^{\oplus n}$ for some $n\ge 1$,
then $\phi_M$ is injective.
\end{enua}
\end{prp}
\begin{proof}
We only prove (1) since (2) follows from a similar discussion.
By the definition of $\phi_M$, each elements in $\Ker(\phi_M)$ is of the form
$\begin{bsmatrix}
0 & 0 \\
0 & g
\end{bsmatrix}
\in \ZZ_R(M\oplus N)$.
Suppose $g\not=0$.
By the assumption, there is a morphism $h\colon M \to N$ such that $g\circ h\not=0$.
Then we have
\[
\begin{bmatrix}
0 & 0 \\
g \circ h & 0
\end{bmatrix}
=\begin{bmatrix}
0 & 0 \\
0 & g
\end{bmatrix}
\begin{bmatrix}
0 & 0 \\
h & 0
\end{bmatrix}=
\begin{bmatrix}
0 & 0 \\
h & 0
\end{bmatrix}
\begin{bmatrix}
0 & 0 \\
0 & g
\end{bmatrix}
=\begin{bmatrix}
0 & 0 \\
0 & 0
\end{bmatrix}.
\]
This shows that $g\circ h=0$, a contradiction.
Consequently, $\Ker(\phi_M)=0$.
\end{proof}

%
%

\begin{cor}\label{prp:Z(R^n oplus M)}
Let $M\in \catmod R$.
If $F\in \catmod R$ is a nonzero $R$-free module, then the natural $R$-algebra homomorphism $R \to \ZZ_R(F\oplus M)$ is an isomorphism.
\end{cor}
\begin{proof}
In this case, 
the map $\phi_{F} \colon \ZZ_R(F \oplus M) \to \ZZ_R(F)$ is injective
by the assumption and \cref{prp:Z(M oplus N) to Z(M) inj}.
The corollary follows from the following commutative diagram of $R$-algebras:
\[
\begin{tikzcd}
\ZZ_R(F \oplus M) \arr{r,"{\phi_F}",hook} & \ZZ_R(F) \\
R \arr{u} \arr{ru,"\iso"'} &.
\end{tikzcd}
\]
\end{proof}

\subsection{Localization of the centers of endomorphisms algebras}
We study local properties of the centers of endomorphisms algebras.
The main result of this subsection is \cref{characterize free center},
which gives a local criterion for the center of the endomorphism algebra to be isomorphic to the ring $R$.

Let us begin with the commutativity of taking centers and localization.
Recall that a \emph{Noether $R$-algebra} $\LL$ is a (possibly noncommutative) $R$-algebra which is finitely generated as an $R$-module.
Note that $R\cdot 1_{\LL}$ is contained in the center of $\LL$ by the definition of $R$-algebras.
For example, the endomorphism algebra $\End_R(M)$ is a Noether $R$-algebra for any $M\in \catmod R$.
\begin{prp} \label{prp:local center}
Let $\LL$ be a Noether $R$-algebra and $S$ a multiplicatively closed subset of $R$.
Then we have a natural $R_S$-algebra isomorphism $Z(\LL)_S \isoto Z(\LL_S)$.
\end{prp}
\begin{proof}
We can easily see that a natural map $\LL \to \LL_S$ induces an injective $R_S$-algebra homomorphism $Z(\LL)_S \to Z(\LL_S)$.
We identify $Z(\LL)_S$ as a subalgebra of $Z(\LL_S)$ by this injection.
Take a system of generators $g_1,\dots,g_n$ of the $R$-module $\LL$.
Then for an element $f/s\in \LL$, we have
\begin{align*}
f/s\in \ZZ(\LL_S)
\iff&\text{for any $g/t\in \LL_S$, $f/s\cdot g/t=g/t\cdot f/s$, }\\
\iff&\text{for any $i=1,\dots,n$, $f/s\cdot g_i/1=g_i/1\cdot f/s$, }\\
\iff&\text{for any $i=1,\dots,n$, there exists $u_i\in S$ such that $u_is(fg_i-g_if)=0$ in $\LL$,}\\
\iff&\text{there exists $u\in S$ such that for any $i=1,\dots,n$, $u(fg_i-g_if)=0$ in $\LL$,}\\
\iff&\text{there exists $u\in S$ such that for any $g\in \LL$, $ufg-guf=0$ in $\LL$,}\\
\iff&\text{there exists $u\in S$ such that $uf\in \ZZ(\LL)$,}\\
\iff&f/s=uf/us\in \ZZ(\LL)_S.
\end{align*}
This proves $Z(\LL_S)=Z(\LL)_S$.
\end{proof}

\begin{cor}
Let $M\in \catmod R$ and $S$ a multiplicatively closed subset of $R$.
Then we have a natural $R_S$-algebra isomorphism $\ZZ_{R}(M)_S \isoto \ZZ_{R_S}(M_S)$.
\end{cor}
\begin{proof}
It follows from \cref{prp:local center} and $\End_R(M)_S \iso \End_{R_S}(M_S)$.
\end{proof}

Next, we study when $\ZZ_R(M)\iso R$ holds.
The key is a detailed inspection of the canonical homomorphism $R \to \ZZ_R(M)$.
We need some preparation for this.
The following is a basic and well-known fact about endomorphisms of finitely generated modules.

\begin{lem}[Determinantal trick] \label{det trick}
Let $R$ be a commutative ring, $I$ an ideal, $M$ a finitely generated $R$-module, and $f\colon M \to M$ an $R$-homomorphism such that $\Ima f\subseteq IM$.
Then there exist an integer $n\ge 1$ and elements $a_i\in I^i$ such that
\[
f^n+a_1f^{n-1}+\cdots+a_n=0.
\]
In particular, $f^n\in I(Rf+R\id_M)$.
\end{lem}

\begin{proof}
See \cite[2.1.8]{HS}.
\end{proof}


We now study the base case $\depth M =0$.
\begin{lem} \label{center over depth zero}
Let $(R,\mm)$ be a commutative local ring and $M\in \catmod R$.
Assume $\depth M=0$.
\begin{enua}
\item
For any $f\in \ZZ_R(M)$, there is $u\in R$ such that $\Ima(f-u\cdot\id_M)\subseteq \mm M$.
\item 
For any $f\in \ZZ_R(M)$, there is $u\in R$ such that $f-u\cdot\id_M\subseteq \mathrm{Jac}(\ZZ_R(M))$.
\item $\ZZ_R(M)$ is a local ring.
\item Suppose that $\depth R = 0$ and $\ZZ_R(M)$ is $R$-free.
Then $\ZZ_R(M)\cong R$.

\end{enua}
\end{lem}
\begin{proof}
(1) 
We may assume that $\Ima f\not\subseteq \mm M$ since there is nothing to prove 
when $\Ima f\subseteq \mm M$.
Therefore the induced homomorphism $\overline{f}\colon M/\mm M \to M/\mm M$ is nonzero.
What we need to show is that $\overline{f}=u\cdot\id_{M/\mm M}$ for some $u\in R/\mm$.
Assume contrary.
Then there exists $v\in M/\mm M$ such that $v$ and $\overline{f}(v)$ are linearly independent over $R/\mm$.
Note that $\soc M \ne 0$.
Let $g\colon M/\mm M \to \soc M$ be an $R/\mm$-linear map such that $g(v)=0$ and $g(\overline{f}(v))\not=0$.
Consider the composition $h\colon M \surj M/\mm M \xrightarrow[]{g} \mathrm{soc}(M) \inj M$. 
Take a preimage $w\in M$ of $v$.
Then $h(w)=0$ and $h(f(w))\not=0$.
On the other hand, since $f$ is in the center of $\End_R(M)$, $f\circ h=h\circ f$.
Thus we have $0=f(h(w))=h(f(w))\not=0$, a contradiction. 

(2)
Take $f\in \ZZ_R(M)$.
By (1), $\Ima(f-u \cdot\id_M)\subseteq \mm M$ for some $u\in R$.
Then, by \cref{det trick}, 
\[
(f-u \cdot\id_M)^n\in \mm (f,\id_M)\subseteq \mm\ZZ_R(M)\subseteq \mathrm{Jac}(\ZZ_R(M))
\]
 for some $n\ge 1$.
See \cite[Corollary 5.9]{Lam} for the last inclusion.
As the Jacobson radical of a commutative ring is a radical ideal, 
we obtain $(f-u \cdot \id_M)\in \mathrm{Jac}(\ZZ_R(M))$.

(3) It follows from (2) that the composition $R \to \ZZ_R(M) \to \ZZ_R(M)/\mathrm{Jac}(\ZZ_R(M))$ is a surjection.
Therefore, $\ZZ_R(M)/\mathrm{Jac}(\ZZ_R(M))$ is local, which yields that $\ZZ_R(M)$ is also local.

(4)
By (1), $\Ima(f-u \cdot\id_M)\subseteq \mm M$ for some $u\in R$.
Thus $f-u\cdot\id_M$ is annihilated by $\soc R$.
It follows that $f-u \cdot\id_M \in \mm\ZZ_R(M)$ since $\ZZ_R(M)$ is nonzero $R$-free.
Hence, we have $\ZZ_R(M)=R \cdot\id_M + \mm \ZZ_R(M)$.
By Nakayama's lemma, $\ZZ_R(M)$ is generated by $\text{id}_M$ as an $R$-module.
In particular, $\ZZ_R(M)$ is cyclic over $R$, and hence $\ZZ_R(M)\cong R$.
\end{proof}

\begin{rmk}
Let $R$ be a commutative henselian local ring and $M\in \catmod R$.
\begin{enua}
\item Assume $M$ is indecomposable. 
Then $\End_R(M)$ has no nontrivial idempotent, and thus so does $\ZZ_R(M)$.
Since $R$ is henselian, it follows that $\ZZ_R(M)$ is local.

\item The assumption that either $\depth M=0$ or $M$ is indecomposable is indispensable to ensure that $\ZZ_R(M)$ is local.
For example, let $k$ be a field, $R=k[\![x,y]\!]/(xy)$, and $M=R/(x)\oplus R/(y)$.
Then $\End_R(M)=\ZZ_R(M)\cong R/(x)\times R/(y)$.
\end{enua}
\end{rmk}

We need the following lemmas to give a local criterion for $R \iso \ZZ_R(M)$
(\cref{characterize free center}).
\begin{lem}\label{prp:split mono alg hom}
Let $\phi \colon R \to S$ be an $R$-algebra homomorphism.
If $S$ is finitely generated projective as an $R$-module and $\Supp_R S = \Spec R$,
then $\phi$ is a split monomorphism of $R$-modules,
and hence $\Cok(\phi)$ is also a finitely generated projective $R$-module.
\end{lem}
\begin{proof}
It is enough to show that $\phi\otimes_R \kk(\pp) \colon \kk(\pp) \to S\otimes_R \kk(\pp)$ is injective for any $\pp\in \Spec R$ by \cite[Proposition 8.10, (i)$\equi$(iii)]{GW-AG}.
It is equivalent to $S\otimes_R \kk(\pp) \ne 0$ for all $\pp\in \Spec R$
as $\phi\otimes_R \kk(\pp)$ is a $\kk(\pp)$-algebra homomorphism,
and hence preserves the multiplicative identity.
This follows from $\Supp_R S = \Spec R$ and Nakayama's lemma.
\end{proof}

\begin{lem}\label{prp:Supp Z(M)}
Let $M\in \catmod R$.
The following hold.
\begin{enua}
\item
$\Supp_R \ZZ_R(M) =\Supp_R M$ for any $M\in\catmod R$.
\item $\depth_{R_\pp} \ZZ_{R_\pp}(M_\pp) \ge \min\{2,\depth_{R_\pp} M_\pp\}$ for any $\pp\in\Spec R$.
\item
$\Ass_R \ZZ_R(M) \subseteq \Ass_R M$ for any $M\in\catmod R$.
\item
$\Ass_R P = \Ass R$
for any $P\in \proj R$ such that $\Supp P=\Spec R$.
\end{enua}
\end{lem}
\begin{proof}
(1):
For any $\pp \in \Spec R$, we have that
\[
M_{\pp}\ne 0 \iff 0 \ne \id_{M_{\pp}}\text{ in }\ZZ_{R_{\pp}}(M_{\pp}) \iff \ZZ_R(M)_{\pp} \ne 0.
\]
This yields the desired equality.

(2): Let $\pp$ be a prime ideal of $R$.
Consider the following subcategory of $\catmod R$:
\[
\catX=\{X \in \catmod R\mid \depth_{R_\pp} X_\pp \ge \min\{2,\depth_{R_\pp} M_\pp\} \}.
\]
By the depth lemma, $\catX$ is KE-closed in $\catmod R$.
Obviously, $M$ belongs to $\catX$.
Therefore, $\ZZ_R(M)\in \catX$ by \cref{fct:Hom-ideal}.
This means that $\depth_{R_\pp} \ZZ_{R_\pp}(M_\pp) \ge \min\{2,\depth_{R_\pp} M_\pp\}$.

(3):
It easily follows from (2).

(4):
For any $\pp \in \Spec R$, we have that $P_{\pp} \iso R_{\pp}^{\oplus n}$ 
for some positive integer $n>0$.
Thus, we have $\Ass_{R_{\pp}} (P_{\pp})=\Ass_{R_{\pp}} (R_{\pp})$ for all $\pp \in \Spec R$,
and $\Ass P = \Ass R$.
\end{proof}

The following proposition shows that the freeness of $\ZZ_R(M)$ is detected by the localizations at primes of depth at most one.

\begin{prp} \label{characterize free center}
Let $R$ be a commutative ring and $M\in \catmod R$.
Then the following are equivalent.
\begin{enua}
\item The natural homomorphism $\iota\colon R \to \ZZ_R(M)$ is an isomorphism.
\item $\ZZ_R(M)\cong R$ as $R$-algebras.
\item $\ZZ_R(M)\cong R$ as $R$-modules.
\item $M\not=0$ and $\ZZ_R(M)$ is $R$-free.
\item $\Supp M=\Spec R$ and $\ZZ_R(M)$ is $R$-projective.
\item $\Ass_R \ZZ_R(M)=\Ass R$ and $\ZZ_{R_\pp}(M_\pp)$ is $R_\pp$-free for all $\pp\in \Spec R$ such that $\depth R_\pp\le 1$.
\end{enua}
\end{prp}

\begin{proof}
The implications (1)$\Rightarrow$(2)$\Rightarrow$(3) $\Rightarrow$(4) are obvious.
The implications (4)$\Rightarrow$(5)$\Rightarrow$(6) easily follow from \cref{prp:Supp Z(M)}.

(6)$\Rightarrow$(1):
Let $C$ be the cokernel of $\iota$.
Applying \cref{prp:split mono alg hom}, we have that $\iota_{\pp}$ is a split monomorphism and $C_{\pp}$ is $R_{\pp}$-free 
for any $\pp\in \Spec R$ such that $\depth R_\pp\le 1$.
In particular, the homomorphism $\iota_{\pp}$ is injective for any $\pp \in \Ass R$, and hence $\iota$ is injective.
It is enough to show that $C_\pp=0$ for any $\pp\in \Spec R$.
We prove this by induction on $\htt \pp$.
Suppose that $\htt \pp = 0$.
Then $\depth R_\pp = \depth_{R_\pp} M_\pp = 0$ and $\ZZ_{R_\pp}(M_\pp)$ is $R_\pp$-free by the assumption.
Therefore, by \cref{center over depth zero}, $\ZZ_{R_\pp}(M_\pp)\cong R_\pp$.
Hence the rank of the free $R_\pp$-module $C_\pp$ is zero, that is, $C_{\pp}=0$.
Next, suppose that $\htt \pp\ge 1$.
Then $C_\pp$ has finite length over $R_\pp$ by the induction hypothesis.
If $\depth R_\pp \le 1$, then $C_\pp$ is $R_\pp$-free and of finite length. Hence $C_\pp=0$ as $\dim R_\pp \ge 1$.
Thus, we may assume that $\depth R_\pp \ge 2$.
Note that $\pp \not\in \Ass R=\Ass(\ZZ_R(M))$ by the assumption.
In other words, $\depth_{R_\pp}\ZZ_{R_\pp}(M_\pp)\ge 1$.
Applying the depth lemma to the exact sequence $0 \to R_\pp \to \ZZ_{R_\pp}(M_\pp) \to C_\pp \to 0$, we derive that $C_\pp=0$, otherwise $\depth R_\pp=1$ as $C_\pp$ has finite length over $R_\pp$, a contradiction.
\end{proof}

Motivated by this proposition,
we introduce non-trivial center loci as follows and study the relationship with non-free loci.
\begin{dfn}
Let $M\in\catmod R$.
\begin{enua}
\item
We call $\NF(M):=\{\pp\in\Spec R\mid M_\pp\text{ is not free over $R_\pp$}\}$
the \emph{non-free locus} of $M$.
Note that $\NF(M)$ is closed in $\Spec R$.
\item
We call $\NZ(M):=\{\pp\in\Spec R\mid \text{$R_{\pp} \not\iso \ZZ_{R_\pp}(M_\pp)$ as $R_\pp$-algebras}\}$
the \emph{non-trivial center locus} of $M$.
\end{enua}
\end{dfn}

The following two lemmas about $\NZ(M)$ will be used in the next section.

\begin{lem}\label{prp:NZ NF}
Let $M \in \catmod R$ such that $\Supp M =\Spec R$.
\begin{enua}
\item
$\NZ(M) =\NF(\ZZ_R(M))\subseteq \NF(M)$ holds.
In particular, $\NZ(M)$ is closed in $\Spec R$.
\item
$\NZ(M \oplus N) \subseteq \NF(M)$ holds
for any $N\in \catmod R$.
\item $\Ass_R(\ZZ_R(M))\subseteq\NZ(M) \cup \Ass R$.
\end{enua}
\end{lem}
\begin{proof}
(1): 
The equality follows from \cref{characterize free center}, and the inclusion follows from \cref{prp:Z(R^n oplus M)}.

(2): 
Take any prime ideal $\pp$ of $R$ such that $\pp \not\in \NF(M)$.
Then $M_{\pp} \iso R_{\pp}^{\oplus n}$ for some positive integer $n>0$
since $\Supp M = \Spec R$.
Then the isomorphism $\ZZ_{R_{\pp}}(M_{\pp} \oplus N_{\pp}) \iso R_{\pp}$ follows from \cref{prp:Z(R^n oplus M)}.
This implies $\pp \not\in \NZ(M\oplus N)$, and thus we have $\NZ(M \oplus N) \subseteq \NF(M)$.

(3): Let $\pp\in \Ass_R\ZZ_R(M) \setminus \NZ(M)$.
Then $\pp \in \Ass_{R_\pp} \ZZ_{R_{\pp}}(M_\pp)=\Ass R_\pp$ as $\ZZ_{R_{\pp}}(M_\pp) \iso R_\pp$.
Thus $\pp\in \Ass R$.
\end{proof}

\begin{lem} \label{NZ and subset}
Let $\catX$ be a KE-closed subcategory of $\catmod R$ and $\Phi$ be a subset of $\Spec R$ containing $\Ass R$.
Assume for each $\pp\in\Phi$, there is an $R$-module $M^\pp\in \catX$ such that $(M^\pp)_\pp\cong R_\pp$.
Then there exists $N\in\catX$ such that $\NZ(N)\cap\Phi=\emptyset$.
\end{lem}

\begin{proof}
By the assumption, for each $\pp\in \Ass R$, there is $M^\pp\in\catX$ such that $(M^\pp)_\pp\cong R_\pp$.
We set $N:=\bigoplus_{\pp\in \Ass R}M^{\pp}\in \catX$.
Clearly, $\Supp N=\Spec R$.
In view of \cref{prp:Z(R^n oplus M)}, the canonical homomorphism $\iota\colon R \to \ZZ_R(N)$ is locally an isomorphism on $\Ass R$.
In particular, $\iota$ is injective.
Thus, $\NZ(N)=\Supp(\Cok \iota)$; see \cref{characterize free center}.
If $\Supp(\Cok \iota)\cap \Phi=\emptyset$, then we are done.
Suppose that there is a prime ideal $\pp\in \Supp(\Cok \iota)\cap \Phi$.
Take $M^\pp\in \catX$ such that $(M^\pp)_\pp\cong R_\pp$.
Set $N':=\ZZ_R(N)\oplus M^{\pp}\in \catX$.
We see that $\Supp N'=\Spec R$ and that
\[
\NZ_R(N')\subseteq \NF_R(N') \subseteq \NF_R(\ZZ_R(N))\setminus\{\pp\}=\NZ_R(N)\setminus\{\pp\}.
\]
Iterating this, we get $N\in\catX$ such that $\NZ(N)\cap \Phi=\emptyset$, otherwise we obtain a strict descending chain $\NZ(N) \supsetneq \NZ(N') \supsetneq \cdots$ of closed subsets of $\Spec R$, which contradicts to that $\Spec R$ is noetherian.
\end{proof}

\section{KE-closed subcategories and higher associated primes}\label{s:KE-closed}

In this section, we investigate KE-closed subcategories and certain subsets of $\Spec R$ associated to them.
We first define a kind of a higher analogue of the set of associated primes for $R$-modules and subcategories.

\begin{dfn}\label{dfn:A^n}
Let $n\ge 0$ be an integer.
\begin{enua}
\item 
For $M\in\catmod R$, we denote by $A^n_R(M)$ the subset $\{\pp\mid \depth_{R_\pp}M_\pp\le n\}$ of $\Spec R$.
\item 
For a subcategory $\catX$ of $\catmod R$, we put $A^n(\catX)=\bigcup_{X\in\catX}A^n(X)$.
\item 
For a subset $\Phi\subseteq \Spec R$, we put $\catmod^n_\Phi R=\{M\in\catmod R\mid A^n(M)\subseteq \Phi\}$.
\end{enua}
\end{dfn}

By definition, the following is immediate.

\begin{rmk} \label{rem n-ass}
Let $M,N\in\catmod R$, $\catX, \catY$ subcategories of $\catmod R$, and $\Phi, \Psi$ subsets of $\Spec R$.
We have the following:
\begin{enua}
\item $\Ass_R(M)=A^0_R(M)$ and $\Supp_R(M)=\bigcup_{n\ge 0}A^n_R(M)$.
\item $A^i_R(M\oplus N)=A^i_R(M)\cup A^i_R(N)$.
\item
For any $\pp\in\Spec R$,
we have that
\[
\pp \in A^n_R(\catX) \Equi \inf_{X\in\catX}\{ \depth_{R_{\pp}} X_{\pp}\} \le n.
\]
\item $A^0(\catX) \subseteq A^1(\catX) \subseteq A^2(\catX) \subseteq \cdots$.
\item If $\catX \subseteq \catY$, then $A^n(\catX) \subseteq A^n(\catY)$.
\item
For any multiplicatively closed set $S$ of $R$,
we have
\[
A^n_{R_S}(M_S)=A^n_R(M) \cap \Spec R_S,\quad
A^n_{R_S}(\catX_S)=A^n_R(\catX) \cap \Spec R_S.
\]
\item $\catmod^0_\Phi R=\catmod^\ass_\Phi R$.
\item $\catmod^1_\Phi R$ is KE-closed.
\item 
$\catmod_\Phi^\ass R \supseteq \catmod_\Phi^1 R \supseteq \catmod_\Phi^2 R\supseteq \cdots$.
\item If $\Phi$ is specialization-closed,
then $\catmod_\Phi^\ass R = \catmod_\Phi^1 R = \catmod_\Phi^2 R = \cdots$.
\item If $\Phi \subseteq \Psi$, then $\catmod_\Phi^n R \subseteq \catmod_\Psi^n R$.
\item $M \in \catmod^n_{\Phi} R$ if and only if 
$M_{\pp} \in \catmod^n_{\Phi\cap\Spec R_\pp} R_{\pp}$ for any $\pp \in \Spec R$.
\end{enua}
\end{rmk}

%
%

A KE-closed subcategory $\catX$ with $R\in \catX$ is reconstructed from the pair $(A^0(\catX), A^1(\catX))$ as follows.

\begin{lem} \label{dom res and A^1}
Let $\catX$ be a subcategory of $\catmod R$.
Set $\Phi=A^0(\catX)$ and $\Psi=A^1(\catX)$.
Then the following are equivalent.
\begin{enua}
\item $\catX$ is a KE-closed subcategory, and $R\in\catX$.
\item $\catX=\catmod^\ass_\Phi R\cap \catmod^1_\Psi R$, $A^0(R)\subseteq \Phi$, and $A^1(R)\subseteq \Psi$.
\end{enua}
\end{lem}

\begin{proof}
This is a translation of \cref{KE-closed res,thm:dom res} in terms of $A^n(\catX)$.
\end{proof}

Next we investigate the behavior of the subsets $A^n(\catX)$ under finite homomorphisms of rinns $R \to S$.

\begin{lem}\label{comparison of depth}
Let $R\to S$ be a finite homomorphism
and $\varphi \colon \Spec S \to \Spec R$ the induced map.
For any $M\in \catmod S$ and $\pp \in \Spec R$, we have
\[
\inf_{\qq\in \varphi^{-1}(\pp)}\left\{\depth_{S_\qq} M_\qq\right\}=\depth_{R_{\pp}} M_{\pp}.
\]
\end{lem}
\begin{proof}
Localizing $R$, $S$ and $M$ at $\pp$, we may assume $(R,\pp)$ is a local ring.
Then the assertion follows by the following
\[
\inf_{\qq\in \varphi^{-1}(\pp)}\{\depth_{S_\qq}M_\qq\} = \inf_{\qq\supseteq \pp S}\{\depth_{S_\qq}M_\qq\} =\grade(\pp S, M)=\grade(\pp,M)=\depth M.
\]
Here the equality $\grade(\pp S, M)=\grade(\pp,M)$ is an immediate consequence of the grade-sensitivity of Koszul complexes \cite[Theorem 1.6.17]{BH}.
\end{proof}

\begin{cor} \label{comparison of A^n}
Let $R \to S$ be a finite homomorphism of commutative noetherian rings.
Let $\varphi\colon \Spec S \to \Spec R$ be the induced map.
Let $M\in \catmod S$.
Then for each $n\ge 0$, 
we have $\varphi(A^n_S(M)) = A^n_R(M)$
and $A^n_S(M) \subseteq \varphi^{-1}(A^n_R(M))$. \qed
\end{cor}

%
%
%

\begin{prp} \label{ass of certain hom}
Let $R \to S$ be a finite homomorphism of commutative noetherian rings.
Let $\varphi\colon \Spec S \to \Spec R$ be the induced map.
For $M\in \catmod R$, we have $\Ass_S(\Hom_R(S,M))=\varphi^{-1}(\Ass_R(M))$.
\end{prp}

\begin{proof}
We have the following by \cref{ass hom} and \cref{comparison of depth},
\[
\Ass_S(\Hom_R(S,M)) \subseteq \varphi^{-1}(\Ass_R(\Hom_R(S,M))) \subseteq \varphi(\Ass_R(M)).
\]
Conversely, take $P\in \Spec S$ such that $\pp:=\varphi(P)\in \Ass_R(M)$.
Localizing $R$, $S$ and $M$ at $\pp$, we may assume $\pp$ (resp.~$P$) is a maximal ideal of $R$ (resp.~$S$).
Moreover, $\Hom_R(S/P,M)\not=0$ since $\pp\in \Ass_R M$ and $S/P$ is only supported on $\pp$.
Then the hom-tensor adjointness isomorphism yields that
\[
\Hom_S(S/P,\Hom_R(S,M))\cong \Hom_R(S/P\otimes_S S,M) \not=0.
\]
As $P$ is maximal, this means that $P\in \Ass_S(\Hom_R(S,M))$.
\end{proof}

\begin{lem}\label{prp:A^1(p)}
For any $\pp \in \Spec R$,
we have $A^1_R(\pp) \subseteq A^1_R(R) \cup \{\pp\}$.
\end{lem}
\begin{proof}
Take $\qq \in A^1_R(\pp)$.
Then we have $\depth_{R_{\qq}} \pp_{\qq} \le 1$.
This occurs only when $\depth R_{\qq} \le 1$ or $\qq=\pp$ by \cref{prp:depth p}.
Thus, we have $\qq \in A^1_R(R) \cup \{\pp\}$.
\end{proof}

\begin{prp}\label{prp:prime in dom res KE}
Let $\catX$ be a KE-closed subcategory of $\catmod R$ containing $R$,
and let $\pp \in \Spec R$.
Then $\pp \in \catX$ if and only if $\pp \in A^1_R(\catX)$.
\end{prp}
\begin{proof}
Suppose that $\pp \in\catX$.
If $\depth R_{\pp} \ge 1$,
then $\inf_{X\in\catX} \{\depth_{R_{\pp}} X_{\pp}\} \le \depth_{R_\pp} \pp_\pp = 1$
by \cref{prp:depth p}.
If $\depth R_{\pp} = 0$, 
then $\inf_{X\in\catX} \{\depth_{R_{\pp}} X_{\pp}\} \le \depth R_\pp = 0$.
In both cases, we have $\pp \in A^1_R(\catX)$.

Conversely, suppose that $\pp \in A^1(\catX)$.
It is enough to show that $A^i(\pp) \subseteq A^i(\catX)$ for $i=0,1$ by \cref{dom res and A^1}.
This follows from the following:
\[
A^0(\pp) \subseteq A^0(R) \subseteq A^0(\catX),\quad
A^1(\pp) \subseteq A^1(R) \cup \{\pp\} \subseteq A^1(\catX).
\]
Here, we use the fact that $\pp \subseteq R$ and \cref{prp:A^1(p)}.
\end{proof}

The following proposition plays a key role in the proof of the main theorem (\cref{thm:KE-closed reconstruction}) of this section.

\begin{prp} \label{ass change of rings}
Let $R \to S$ be a finite homomorphism of commutative noetherian rings.
Let $\varphi\colon \Spec S \to \Spec R$ be the induced map and $\rho\colon \catmod S \to \catmod R$ the restriction functor.
Then for a KE-closed subcategory $\catX$ of $\catmod R$, the following assertions hold true:
\begin{enua}
\item $\Ass_S(\rho^{-1}(\catX))=\varphi^{-1}(\Ass_R(\catX))$.
\item If $S_R\in\catX$, then
$A^1_S(\rho^{-1}(\catX))=\varphi^{-1}(A^1_R(\rho\rho^{-1}\catX))$.
\end{enua}
\end{prp}

\begin{proof}
(1): 
Note that if $X\in \catX$, then $\Hom_R(S,X) \in \rho^{-1}(\catX)$. Then
the inclusion $\varphi^{-1}(\Ass_R(\catX))\subseteq \Ass_S(\rho^{-1}(\catX))$ follows from \cref{ass of certain hom}, and the opposite one is due to \cref{comparison of A^n}.

(2): 
It follows from \cref{comparison of A^n} that $A^1_S(\rho^{-1}(\catX))\subseteq \varphi^{-1}(A^1_R(\rho\rho^{-1}\catX))$.
We prove the opposite inclusion.
Let $P\in \varphi^{-1}(A^1_R(\rho\rho^{-1}\catX))$ 
and set $\pp:=\varphi(P) \in A^1_R(\rho\rho^{-1}\catX)$.
We aim to show $P\in A^1_S(\rho^{-1}(\catX))$.
There is $N\in \catmod S$ such that $N_R \in \catX$ and $\pp \in A^1_R(N)$.
\cref{comparison of depth} says that there is $Q\in \varphi^{-1}(\pp)$ such that 
$\depth_{S_Q} N_Q = \depth_{R_\pp} N_\pp \le 1$.
This means $Q \in A^1_S(\rho^{-1}(\catX))$.
By the assumption $S_R\in\catX$ and \cref{prp:prime in dom res KE},
we have $Q\in \rho^{-1}(\catX)$.
In the following, we prove that $\rho^{-1}(\catX)$ also contains $P$.

We have $(S/P)_\pp \cdot \pp=0$ as $PS_{\pp} \supseteq \pp S_{\pp}$.
Thus $(S/P)_{\pp} \iso \kk(\pp)^{\oplus m}$ as $R_{\pp}$-modules for some $m > 0$.
Similarly, we have $(S/Q)_{\pp} \iso \kk(\pp)^{\oplus n}$ as $R_{\pp}$-modules for some $n > 0$.
Thus, $(S/P)^{\oplus n}_{\pp} \iso \kk(\pp)^{\oplus nm} \iso (S/Q)^{\oplus m}_{\pp}$
as $R_{\pp}$-modules.
Extending this $R_{\pp}$-isomorphism, we obtain an $R$-linear map $f \colon (S/P)^{\oplus n} \to (S/Q)^{\oplus m}$ such that $f_{\pp}$ is an isomorphism.
Then $f$ is injective since $\Ass_R(S/P)=\varphi(\Ass_S(S/P)) =\{\pp\}$.
Composing the homomorphism $f$ and the natural surjection $S^{\oplus n} \surj (S/P)^{\oplus n}$,
we obtain an exact sequence $0 \to P^{\oplus n} \to S^{\oplus n} \to (S/Q)^{\oplus m}$.
Pulling back this sequence 
by the natural surjection $S^{\oplus m} \surj (S/Q)^{\oplus m}$,
we obtain the following commutative diagram of $R$-modules with exact rows and columns:

\[
\begin{tikzcd}
& & 0 \ar[d] & 0 \ar[d] &\\
& & Q^{\oplus m} \ar[d] \ar[r, equal] & Q^{\oplus m} \ar[d] &\\
0 \ar[r] & P^{\oplus n} \ar[r] \ar[d, equal] & X \ar[r] \ar[d] \pb & S^{\oplus m}  \ar[d] \\
0 \ar[r] & P^{\oplus n} \ar[r] & S^{\oplus n} \ar[r] \ar[d] & (S/Q)^{\oplus m}  \ar[d] \\
& & 0 & 0 &.
\end{tikzcd}
\]
Since $Q_R,S_R\in\catX$ and $\catX$ is extension-closed, we have $X\in\catX$.
Also, since $S_R, X\in\catX$ and $\catX$ is KE-closed, we obtain $P_R\in\catX$,
which means $P\in \rho^{-1}(\catX)$.
Therefore, we have $P\in A^1_S(\rho^{-1}(\catX))$ by \cref{prp:prime in dom res KE}.
\end{proof}

For the proof the main theorem of this section, we need the following three technical lemmas on KE-closed subcategories.

\begin{lem} \label{decomposition}
Let $\catX$ and $\catY$ be KE-closed subcategories of $\catmod R$.
Consider a decomposition $\Ass \catX=\Phi\sqcup \Psi$ as a set.
Assume $\catmod^\ass_\Phi R\subseteq \catX\subseteq \catY$.
If $\catX \cap \catmod^\ass_\Psi R=\catY \cap \catmod^\ass_\Psi R$, then $\catX=\catY$.
\end{lem}

\begin{proof}
Take $M\in \catY$.
Applying \cref{KE desired Ass submodule} to $M$, there is an exact sequence $0 \to L \to M \to N \to 0$ such that $N\in \catmod^\ass_\Phi R$ and $L\in \catY\cap \catmod^\ass_\Psi R$.
By the assumption, $N,L\in \catX$, and so $M\in \catX$.
\end{proof}

\begin{lem} \label{lem compare ass}
Let $\catX$ be a KE-closed subcategory of $\catmod R$.
Set $\catY=\catmod_{A^0(\catX)}^0 R \cap \catmod_{A^1(\catX)}^1 R$.
Then
\begin{enua}
\item $\catX\subseteq \catY$.
\item $A^0_R(\catX)=A^0_R(\catY)$ and $A^1_R(\catX)=A^1_R(\catY)$.
\item If $S\subseteq R$ is a multiplicatively closed subset, then $A^0_{R_S}(\catX_S)=A^0_{R_S}(\catY_S)$ and
$A^1_{R_S}(\catX_S)=A^1_{R_S}(\catY_S)$.
\end{enua}
\end{lem}

\begin{proof}
(1): Let $X\in \catX$.
Then, for $i=0,1$, $A^i(X)\subseteq A^i(\catX)$.
Therefore, $X\in \catY$.

(2): It is clear that for $i=0,1$, $A^i(\catY) \subseteq A^i(\catX)$.
On the other hand, by (1), we have that for $=0,1$, $A^i(\catX) \subseteq A^i(\catY)$.

(3): This follows from (2) with \cref{rem n-ass}.
\end{proof}

\begin{lem}\label{lem for replace Ass}
Assume that $(R,\mm)$ is local.
Let $\catX$ be a KE-closed subcategory of $\catmod R$.
Put $\Phi:=\Ass\catX$ and suppose that $\mm \in \Phi$.
If $\fl R \ne \catX$,
then $A^1(\catX \cap \catmod^{0}_{\Phi\setminus\{\mm\}}R)=A^1(\catX)$.
\end{lem}
\begin{proof}
Set $\catX':=\catX \cap \catmod^{0}_{\Phi\setminus\{\mm\}}R$.
As $A^1(\catX')\subseteq A^1(\catX)$ clearly holds, we prove the converse.
We first show that $\mm \in A^1(\catX')$.
Note that $\fl R \subsetneq \catX$ by $\mm\in \Phi$ and \cref{max ideal KE contains fl}.
Take $X \in \catX$ such that $X\not\in \fl R$.
Applying \cref{KE desired Ass submodule} to $X$,
we obtain a nonzero module $M \in \catX$ such that $\depth M \ge 1$.
Then $\mm M \in \catX$ as $M, M/\mm M \in \catX$.
We also have $\Ass (\mm M) \subseteq \Ass M \subseteq \Phi \setminus \{\mm\}$ and $\depth \mm M = 1$ by the depth lemma.
This proves $\mm \in A^1(\mm M) \subseteq A^1(\catX')$.

Next,
take $\pp \in A^1(\catX)$, and prove $\pp\in A^1(\catX')$.
We may assume that $\pp \ne \mm$.
There exists $X \in \catX$ such that $\pp\in A^1(X)$.
If $\mm \not\in \Ass X$, then $X \in \catX'$, and $\pp \in A^1(\catX')$.
Assume that $\mm \in \Ass X$.
The case $\Ass X = \{\mm\}$ does not occur
because $A^1(\fl R)\subseteq \Supp(\fl R)=\{\mm\}$.
There is an exact sequence $0 \to L \to X \to N \to 0 $ such that 
$\Ass L = \Ass X \setminus\{\mm\}$, $\Ass N = \{\mm\}$, and $L \in \catX$ by \cref{KE desired Ass submodule}.
Then $L \in \catX'$ and $L_\pp \iso X_\pp$.
We have $\depth_{R_\pp} L_\pp =\depth_{R_\pp} X_\pp\le 1$, and hence $\pp \in A^1(\catX')$.
\end{proof}

For a commutative ring, we denote by $Q(R)$ the total ring of quotients.
Then for $R$-submodules $I,J$ of $Q(R)$, write $I:J$ to be $\{a\in Q(R)\mid aJ \subseteq I\}$.
It is easy to see that $I:I$ is an $R$-subalgebra of $Q(R)$.
For our purpose, we need the following modified version of \cite[Proposition 4.27]{KS}.

\begin{prp} \label{subalg seq}
Let $R$ be a commutative noetherian ring.
Let $S$ be an $R$-subalgebra of $Q(R)$ such that $S/R$ has finite length as an $R$-module.
Then there exists a sequence of $R$-subalgebras $R=S_0 \subsetneq S_1 \subsetneq \cdots \subsetneq S_n=S$ such that for each $i=0,\dots,n-1$, there exists $\mm_i\in\Max(S_i)\setminus\Ass(S_i)$ such that $S_{i+1}\subseteq \mm_i:\mm_i$.
\end{prp}

\begin{proof}
We prove by induction on the length $l$ of $S/R$.
If $l=0$, then there is nothing to prove.
Suppose $l>0$.
Take a maximal ideal $\mm$ of $R$ which belongs to $\Ass(S/R)$.
By the definition of associated primes, there exists an element $x\in S\setminus R$ such that $\mm x\subseteq R$.
We then have $\mm x \subseteq \mm$; otherwise, the multiplication by $x$ gives a surjection $\mm R_\mm \to R_\mm$. 
This implies that $R_\mm$ is a discrete valuation ring.
As $S$ is integral over $R$, $S_\mm$ must be equal to $R_\mm$ (Here, there is a canonical inclusion from $Q(R)_\mm$ to $Q(R_\mm)$, so that $S_\mm$ can be regarded as a subring of $Q(R_\mm)$).
However, since $\mm\in\Ass(S/R)$, $(S/R)_\mm$ is nonzero, which shows a contradiction.

We see that $x\in \mm:\mm$.
Since $\mm:\mm$ is an $R$-algebra, $S_1:=R[x]$ is a subalgebra of $\mm:\mm$.
Obviously, $S_1$ is not equal to $R$ and contained in $S$.
It follows that the length of the $S_1$-module $S/S_1$ is less than $l$.
By the induction hypothesis, we get a sequence $S_1\subsetneq S_2 \subsetneq\cdots\subsetneq S_n=S$ of $S_1$-algebras with suitable maximal ideals.
Thus, we achieve the desired sequence $R=S_0 \subsetneq S_1 \subsetneq \cdots \subsetneq S_n=S$ of $R$-algebras with suitable maximal ideals.
\end{proof}

Now we achieve the main theorem of this section which asserts that every KE-closed subcategory $\catX$ of $\catmod R$ is reconstructed from the pair $(A^0(\catX),A^1(\catX))$ of subsets of $\Spec R$.

\begin{thm} \label{thm:KE-closed reconstruction}
Let $\catX$ be a KE-closed subcategory of $\catmod R$.
Then $\catX=\catmod_{A^0(\catX)}^\ass R \cap \catmod_{A^1(\catX)}^1 R$.
\end{thm}

\begin{proof}
Set $\Phi=A^0(\catX)$, $\Psi=A^1(\catX)$, and $\catY=\catmod_{\Phi}^\ass R \cap \catmod_{\Psi}^1 R$.
As a first step, we assume that $R$ is local with a maximal ideal $\mm$, and proceed by induction on $\dim R$.
If $\dim R\le 1$, then the assertion follow from \cref{KE=torf in 1-dim}.
We may assume $\dim R\ge 2$.
Let $\pp\in \Spec R\setminus\{\mm\}$.
\cref{rem n-ass,lem compare ass} yield that 
\[
A^0(\catX_\pp)=\Phi\cap \Spec R_\pp=A^0(\catY_\pp),\text{ and }A^1(\catX_\pp)=\Psi\cap \Spec R_\pp=A^1(\catY_\pp).
\]
By the induction hypothesis, it follows that
\begin{equation} \label{eq_a66}
\catX_\pp = \catmod_{\Phi \cap \Spec R_\pp}^0(R_\pp) \cap \catmod_{\Psi \cap \Spec R_\pp}^1(R_\pp)=\catY_\pp.
\end{equation}
Also, by \cref{decomposition,restrict KE by torf,max ideal KE contains fl,lem for replace Ass}, we may assume $\mm\not\in \Phi$.
Take a nonzero module $N\in \catY$.
We aim to show $N\in \catX$.
By \eqref{eq_a66}, $N_\pp$ belongs to $\catX_\pp$ for all $\pp\in \Spec R \setminus\{\mm\}$.
Therefore, by \cref{loc split mono}, there exists $M\in \catX$ such that $N_\pp$ is a direct summand of $M_\pp$ for all $\pp\in \Spec R \setminus\{\mm\}$.
In particular, for $i=0,1$ and $\pp\in \Spec R\setminus \{\mm\}$,
\[
A^i_R(N)\cap \Spec R_\pp=A^i_{R_\pp}(N_\pp)\subseteq A^i_{R_\pp}(M_\pp)=A^i_R(M)\cap \Spec R_\pp.
\]
Therefore, if $\mm\not\in A^1_R(N)$, then for $i=0,1$, $A^i_R(N) \subseteq A^i_R(M)$.
If $\mm\in A^1_R(N)$, then $\mm\in A^1_R(\catY)=A^1_R(\catX)$, and so there exists $M'\in \catX$ such that $\mm\in A^1_R(M')$.
Replacing $M$ with $M\oplus M'$ (which is in $\catX$), we have $A^i_R(N) \subseteq A^i_R(M)$ for $i=0,1$.
Set $S=\ZZ_R(N\oplus M)$.
Note that, since $N_\pp$ is a direct summand of $M_\pp$, $S_\pp \cong \ZZ_{R_\pp}(M_\pp) \cong (\ZZ_R(M))_\pp$ for any $\pp\in \Spec R\setminus \{\mm\}$ (\cref{prp:isom_center_on_add_equiv}).
Also, note that $M$ and $N$ can be regarded as $S$-modules in a natural way (cf.\ \cref{prp:Z(M oplus N)}).
This particularly means that $M_R\in \rho\rho^{-1}(\catX)$, where $\rho\colon \catmod S \to \catmod R$ is the restriction functor.
Moreover, $\ZZ_R(M)$ can be regarded as an $S$-algebra in a natural way.
In particular, $\ZZ_R(M)_S\in \rho^{-1}(\catX)$.

\begin{claim}
$S\in \rho^{-1}(\catX)$.
\end{claim}
Once we prove this claim, then due to \cref{comparison of A^n,ass change of rings}, we have
\[
A^0_S(N)\subseteq \varphi^{-1}(A^0_R(N))\subseteq  \varphi^{-1}(A^0_R(\catX))=A^0_S(\rho^{-1}(\catX))
\]
and
\[
A^1_S(N)\subseteq \varphi^{-1}(A^1_R(N))\subseteq \varphi^{-1}(A^1_R(M))\subseteq \varphi^{-1}(A^1_R(\rho\rho^{-1}\catX))=A^1_S(\rho^{-1}(\catX)),
\]
where $\varphi\colon \Spec S \to \Spec R$ is the induced map.
Hence $N\in\rho^{-1}(\catX)$ by \cref{dom res and A^1}.

Next we prove Claim 1.
Let $P\in\Spec S\setminus \Max S$ and $\pp=\varphi(P)\in \Spec R$.
Then $\pp\not=\mm$, and hence 
\[
S_P\cong (S_\pp)_{PS_\pp} \cong ((\ZZ_R(M))_\pp)_{PS_\pp} \cong \ZZ_R(M)_{P}.
\] 
This implies that $S_P\in (\rho^{-1}(\catX))_P$.
We also have the following by \cref{prp:Supp Z(M)} (3) and the fact that $\mm \not\in \Phi$:
\[
\Ass_S(S)\subseteq \varphi^{-1}(\Ass_R(S))\subseteq \varphi^{-1}(\Spec R\setminus \{\mm\})=\Spec S\setminus \Max S.\]
Using \cref{NZ and subset}, we get an $S$-module $L'\in \rho^{-1}(\catX)$ such that $\NZ_S(L') \subseteq \Max S$.
Then, observe that
\[
\Ass_S(S) \subseteq \Spec S\setminus \Max S \subseteq \Spec S\setminus \NZ_S(L') \subseteq \Supp_S(L').
\]
It follows that $\Supp_S(L')=\Spec S$.
Set $L:=\ZZ_S(L')\oplus M \in \rho^{-1}(\catX)$.
Then, noting that $\Supp L =\Spec S$ by \cref{prp:Supp Z(M)} (1), we obtain the following from \cref{prp:NZ NF}: 
\[\NZ_S(L) \subseteq \NF_S(\ZZ_S(L'))=\NZ_S(L')\subseteq \Max S.
\]
In particular, $\NZ_S(L)\cap \Ass_S(S)=\emptyset$.
This implies that the canonical homomorphism $\iota\colon S \to T:=\ZZ_S(L)$ is locally an isomorphism on $\Ass S$, and hence $\iota$ is injective.
Since $L\in \rho^{-1}(\catX)$, we have that
\[
\Ass_S(T)\subseteq A^0_S(\rho^{-1}(\catX))\subseteq \varphi^{-1}(A^0_R(\catX)) \subseteq \varphi^{-1}(\Spec R\setminus\{\mm\})=\Spec S\setminus\Max S. 
\]
Thus $\Ass_S(T) \cap \NZ_S(L)=\emptyset$.
Then \cref{prp:NZ NF} shows that $\Ass_S(T)\subseteq \Ass S$.
Also, as $\iota$ is injective, $\Ass_S (S)\subseteq \Ass_S(T)$.
Thus we obtain that $\Ass_S (S)=\Ass_S(T)$.
Now suppose that $\depth_R (N\oplus M)\ge 2$.
In other words, $\mm\not\in A^1_R(N\oplus M) =A^1_R(M)$.
Then for all $\nn\in \Max S$,
\[
\depth S_\nn \ge \depth_R S \ge \min\{2,\depth_R(N\oplus M)\}\ge 2
\]
since $S=\ZZ_R(N\oplus M)$; see \cref{prp:Supp Z(M)}.
Thus $\NZ_S(L)\subseteq \Max S \subseteq \{\qq\in\Spec S\mid \depth S_\qq \ge 2\}$.
Applying \cref{characterize free center}, we have $S\cong T$.
Consequently, $S$ belongs to $\rho^{-1}(\catX)$.

We deal with the remaining case, that is, the case where $\depth_R (N\oplus M)\le 1$.
It follows that $\mm\in A^1_R(N\oplus M) \subseteq A^1_R(M)$.
We see that the canonical homomorphism $T \to T\otimes_S Q(S)(\cong Q(S))$ is injective since $\Ass_S(T)=\Ass_S(S)$.
This means that $T$ is an $S$-subalgebra of $Q(S)$.
Using \cref{subalg seq}, we get a sequence $S=S_0\subsetneq S_1\subsetneq \cdots\subsetneq S_m=T$ of $S$-subalgebras of $T$ and a maximal ideal $\mm_i\in \Max(S_i)$ such that $S_{i+1}\subseteq \mm_i:\mm_i$ for each $i$.
Let $\rho_i\colon \catmod S_i \to \catmod R$ be the restriction functor.
For each $i=0,\dots, m$, we claim that 
\begin{claim}
$S_i\in \rho_i^{-1}(\catX)$.
\end{claim}
By letting $i=0$, this implies Claim 1.
In the following, we prove Claim 2 by induction on $i$.
If $i=m$, then $S_i=T$ and there is nothing to prove.
Assume $i<m$.
By the induction hypothesis, $S_{i+1}\in \rho_{i+1}^{-1}(\catX)$.
Note that $M$ has a $T$-module structure, and thus it has an $S_{i+1}$-module structure.
This means that $M_R\in \rho_{i+1}\rho_{i+1}^{-1}(\catX)$.
Thus, \cref{ass change of rings} yields that 
\begin{equation}\label{eq:KE-closed reconstruction 2}
A^1_{S_{i+1}}(\rho_{i+1}^{-1}(\catX))=\varphi_{i+1}^{-1}(A^1_R(\rho_{i+1}\rho_{i+1}^{-1}(\catX)))\supseteq \varphi_{i+1}^{-1}(A^1_R(M))\supseteq \varphi^{-1}_{i+1}(\mm)=\Max S_{i+1}.
\end{equation}
Here $\varphi_{i+1}\colon \Spec S_{i+1} \to \Spec R$ is the morphism corresponding to the canonical $R$-algebra homomorphism.
Observe that $\mm_i$ has an $\mm_i:\mm_i$-module structure via the natural way, and that $\mm_i$ has an $S_{i+1}$-module structure via the inclusion $S_{i+1} \to \mm_i:\mm_i$.
The length of cokernels of two inclusions $\mm_i \to S_i$ and $S_i \to S_{i+1}$ are both finite. Hence, $S_{i+1}/\mm_i$ also has finite length, first as an $S_i$-module and therefore as an $S_{i+1}$-module.
It follows that $(S_{i+1})_Q\cong (\mm_i)_Q$ for any $Q\in \Spec S_{i+1}\setminus \Max S_{i+1}$.
This and \eqref{eq:KE-closed reconstruction 2} imply that 
\[
A^1_{S_{i+1}}(\mm_i)\subseteq A^1_{S_{i+1}}(S_{i+1}) \cup \Max S_{i+1} \subseteq A^1_{S_{i+1}}(\rho_{i+1}^{-1}(\catX)) \cup \Max S_{i+1} \subseteq A^1_{S_{i+1}}(\rho_{i+1}^{-1}(\catX)).
\]
Also, since $\mm_i\subseteq S_{i+1}$, 
we get
\begin{gather*}
A^0_{S_{i+1}}(\mm_i)\subseteq A^0_{S_{i+1}}(S_{i+1})\subseteq A^0_{S_{i+1}}(\rho_{i+1}^{-1}(\catX)).
\end{gather*}
Therefore, $\mm_i\in \rho_{i+1}^{-1}(\catX)$ by \cref{dom res and A^1}, which also means that $\mm_i\in \rho_i^{-1}(\catX)$.
Thanks to \cref{imst}, we conclude that $S_i\in \rho^{-1}_i(\catX)$.

Finally, we consider the case where $R$ is not necessarily local.
For any $\pp\in \Spec R$, we already know that $\catX_\pp=\catY_\pp$.
Therefore, the assertion follows from \cref{loc containment}.
\end{proof}

As an application of this theorem,
we can completely characterize when every KE-closed subcategory is a torsion-free class,
which strengthens the main theorem of \cite{KS}.
See also \cref{prp:KE=torf for S_2 excellent}.
\begin{prp}\label{prp:KE=torf via depth}
The equality $\ke(\catmod R)=\torf(\catmod R)$ holds true if and only if $\depth_{R_\pp}X_\pp\le 1$ for any $X\in \catmod R$ and $\pp \in \Supp X$.
\end{prp}
\begin{proof}
The sufficiency has been proved in \cite[Lemma 4.16]{KS}.
Suppose that $\depth_{R_\pp}X_\pp\le 1$ for any $X\in \catmod R$ and $\pp \in \Supp X$.
Then for any KE-closed subcategory $\catX$, we see that $A^1(\catX)=\Supp \catX$.
Since $\catmod^{\ass}_{\Ass \catX} R \subseteq \catmod^{\ass}_{\Supp\catX} R = \catmod^1_{\Supp\catX} R$,
we have $\catX = \catmod^{\ass}_{\Ass \catX} R$ by \cref{thm:KE-closed reconstruction}.
In particular, $\catX$ is a torsion-free class.
\end{proof}

\begin{ex} \label{ex: 2-dim KE=torf}
There exists a $2$-dimensional commutative noetherian local domain $R$ having no nonzero maximal Cohen-Macaulay $R$-modules; see \cite[Example 2.1]{LW}.
For such a ring $R$, we have $\ke(\catmod R)=\torf(\catmod R)$.
\end{ex}

\section{Bass functions and the classification}\label{s:Bass}
Throughout this section, $R$ is a commutative noetherian ring.
In this section, under a mild assumption, we establish a bijective correspondence between KE-closed subcategories of $\catmod R$ and a certain class of functions on $\Spec R$, which are called \emph{$2$-Bass functions}.

In \S \ref{ss:Bass}, we will introduce Bass functions and establish the bijection under the existence of nonzero $(S_2)$-modules.
In \S \ref{ss:S2}, we will see that nonzero $(S_2)$-modules exist if $R$ is \emph{$(S_2)$-excellent} that is introduced in \cite{Ces}.
It is a mild assumption that is satisfied for homomorphic images of Cohen-Macaulay rings and excellent rings.
In \S \ref{ss:ex Bass}, we give examples of Bass functions and our classification.

\subsection{Bass functions}\label{ss:Bass}
Consider the following assignments:
\begin{itemize}
\item
For a subcategory $\catX$ of $\catmod R$,
define a function $\Spec R \to \bbN \cup \{\infty\}$ by
\[
f_{\catX}(\pp) = \inf_{X\in \catX} \left\{\depth_{R_\pp} X_{\pp}\right\}.
\]
\item
For a function $f\colon \Spec R \to \bbN \cup \{\infty\}$,
define a subcategory of $\catmod R$ by
\[
\catX_f := \{M \in \catmod R \mid \text{$\depth_{R_\pp} M_{\pp} \ge f(\pp)$ for any $\pp \in \Spec R$} \}.
\]
\end{itemize}
These give rise to maps:
\begin{equation}\label{eq:subcat funct corresp}
f_{(-)}:
\{\text{subcategories of $\catmod R$}\} \rightleftarrows \{\text{functions $\Spec R \to \bbN\cup\{\infty\}$}\}
:\catX_{(-)}.
\end{equation}
The left-hand set is naturally ordered by inclusions.
The right-hand set is also ordered by
\[
f\le g \quad :\Longleftrightarrow \quad \text{$f(\pp)\le g(\pp)$ for all $\pp\in\Spec R$.}
\]
Then the maps in \eqref{eq:subcat funct corresp} are order-reversing.
We can easily see that $f_{\catX_f}\ge f$ and $\catX_{f_{\catX}} \supseteq \catX$ hold
for any function $f\colon \Spec R \to \bbN\cup\{\infty\}$ and any subcategory $\catX$ of $\catmod R$.

In what follows, we study the function $f_{\catX}$ associated to a subcategory and the subcategory $\catX_f$ associated to a function, and will prove that the correspondence above gives rise to the bijection between the KE-closed subcategories and the \emph{$2$-Bass functions} introduced in \cref{dfn:Bass fct} (see \cref{cor:bij KE 2-Bass}).

\begin{ex}\label{ex:X_f}
We give examples of subcategories associated to functions.
\begin{enua}
\item
Consider the height function $\htt \colon \pp \to \htt\pp$.
Then $\catX_{\htt}$ is nothing but the category $\cm R$ of maximal Cohen-Macaulay $R$-modules.
\item
For a subset $\Phi$ of $\Spec R$,
define a function on $\Spec R$ as follows:
\[
f_{\Phi}(\pp):=
\begin{cases}
0 & \pp \in \Phi \\
1 & \pp \in \Phi^{\up}\setminus \Phi  \\
\infty & \pp\not\in \Phi^{\up},
\end{cases}
\]
where $\Phi^{\mathsf{up}}:=\{\pp \in \Spec R \mid \exists \qq \in \Phi, \qq \subseteq \pp \}$.
We call $f_\Phi$ the \emph{function associated to the subset} $\Phi$.
Then we can easily see that $\catX_{f_\Phi}=\catmod^{\ass}_{\Phi} R$ (see \cref{fct:Takahashi}).
\end{enua}
\end{ex}

We first see that $\catX_f$ produces KE-closed subcategories under some constraint on the values of $f$.
For a function $f \colon \Spec R \to \bbN \cup \{\infty\}$,
define the \emph{domain of definition} of $f$ by 
\begin{equation}\label{eq:dom}
\dom(f):= \{\pp \in \Spec R \mid f(\pp) < \infty \}.
\end{equation}

\begin{lem}\label{prp:catX_f KE}
Let $f \colon \Spec R \to \bbN \cup \{\infty\}$ be a function.
\begin{enua}
\item
$\catX_f$ is closed under epi-kernels, extensions, and direct summands.
\item
If $f(\pp) = 0$ for any $\pp \in \dom(f)$,
then $\catX_f$ is a Serre subcategory of $\catmod R$.
\item
If $f(\pp) \le 1$ for any $\pp \in \dom(f)$,
then $\catX_f$ is a torsion-free class of $\catmod R$.
\item
If $f(\pp) \le 2$ for any $\pp \in \dom(f)$,
then $\catX_f$ is a KE-closed subcategory of $\catmod R$.
\end{enua}
\end{lem}
\begin{proof}
It follows from the depth lemma.
\end{proof}

Conversely, the functions associated to torsion-free classes and KE-closed subcategories
have some constraints.
We need the following lemma which plays a crucial role throughout this section.
\begin{lem}\label{prp:depth 1 2}
Let $M\in \catmod R$ and $\pp \in \Supp M$.
\begin{enua}
\item
If $\depth_{R_{\pp}}M_\pp \ge 1$, then there is a submodule $N$ of $M$ such that
\[
\depth_{R_\qq} N_{\qq}=
\begin{cases}
\depth_{R_{\qq}} M_{\qq} & \text{if $\qq\not\supseteq \pp$,} \\
1 & \text{if $\qq=\pp$,} \\
\ge \min\{2, \depth_{R_\qq} M_{\qq}\} & \text{if $\qq \supsetneq \pp$.}
\end{cases}
\]
\item
If $\depth_{R_\pp}M_\pp \ge 2$,
then there exists $K \in \catmod R$ such that
\[
\depth_{R_\qq} K_{\qq}=
\begin{cases}
\depth_{R_{\qq}} M_{\qq} & \text{if $\qq\not\supseteq \pp$,} \\
2 & \text{if $\qq=\pp$,} \\
\ge \min\{2,\depth M_\qq\} & \text{if $\qq \supsetneq \pp$.}
\end{cases}
\]
Moreover, for any KE-closed subcategory $\catX$ of $\catmod R$,
if $M\in \catX$, then $K \in \catX$.
\end{enua}
\end{lem}
\begin{proof}
Since $(M/\pp M)_\pp \ne 0$ by Nakayama's lemma,
we have $\Supp_R M/\pp M = V(\pp)$ and $\pp \in \Ass_R(M/\pp M)$.

(1)
There is an exact sequence $0 \to X \to M/\pp M \to Y \to 0$
such that $\Ass_R X = \Ass_R(M/\pp M) \setminus \{\pp\}$ and $\Ass_R Y =\{\pp\}$
by \cref{fct:Goto-Watanabe}.
Then $N:=\Ker(M \surj M/\pp M \surj Y)$ is a desired module 
by the exact sequence $0 \to K \to M \to Y \to 0$ and the depth lemma.
Indeed, we have $\depth_{R_\pp} N_{\pp}=1$ by $\depth_{R_\pp}Y_{\pp}=0$ and $\depth_{R_{\pp}}M_\pp \ge 1$.
If $\qq \not\in V(\pp)$, then $\depth_{R_\qq} N_{\qq}=\depth_{R_{\qq}} M_{\qq}$ as $Y_\qq =0$.
If $\qq \supsetneq \pp$, then $\depth_{R_\qq} Y_{\qq} \ge 1$ as $\qq \not\in \Ass_R Y$.
Thus, we have $\depth_{R_\qq} N_{\qq} \ge \min\{2,\depth M_\qq\}$.

(2)
Take a finite presentation $R^{\oplus n} \xr{f} R \to R/\pp \to 0$.
Observe that $n \ge \htt \pp \ge \depth_{R_\pp} M_\pp \ge 2$.
Tensoring $M$ with this exact sequence,
we obtain the exact sequence $M^{\oplus n} \xr{f\otimes M} M \to M/\pp M \to 0$.
Then $K:=\Ker(f\otimes M)$ is a module with the desired property.
Indeed,
since $\depth_{R_\pp}(M/\pp M)_\pp =0$ and $\depth_{R_\pp} M_\pp \ge 2$, the depth lemma yields $\depth_{R_\pp} K_{\pp} = 2$.
If $\qq \not\in V(\pp)$, 
then $f_\qq$ is a split monomorphism and $\Ker(f_\qq)\iso R_\qq^{\oplus n-1}$.
Thus, we have $K_\qq \iso M_\qq^{\oplus n-1}$.
The case $\qq \supsetneq \pp$ is immediate.
\end{proof}

\begin{lem}\label{prp:bound KE}
Let $\catX$ be a subcategory of $\catmod R$.
\begin{enua}
\item
If $\catX$ is a Serre subcategory, then $f_{\catX}(\pp)\le 0$ for any $\pp \in \Supp \catX$.
\item
If $\catX$ is a torsion-free class, then $f_{\catX}(\pp)\le 1$ for any $\pp \in \Supp \catX$.
\item
If $\catX$ is a KE-closed subcategory, then $f_{\catX}(\pp)\le 2$ for any $\pp \in \Supp \catX$.
\end{enua}
\end{lem}
\begin{proof}
Take any $\pp \in \Supp \catX$, and take $X\in \catX$ such that $\pp \in \Supp X$.

(1)
As $\catX$ is Serre, we have $X/\pp X \in \catX$.
Then $f_{\catX}(\pp)\le \depth_{R_\pp}(X/\pp X)_\pp = 0$.

(2)
If $\depth_{R_\pp} X_{\pp} =0$,
then $f_{\catX}(\pp) \le \depth_{R_\pp} X_{\pp} = 0$.
Suppose $\depth_{R_\pp} X_{\pp} \le 1$.
There is $N\in \catX$ such that $\depth_{R_\pp} N_\pp=1$ by \cref{prp:depth 1 2}.
Here, we use the fact that $\catX$ is closed under submodules.
Then $f_{\catX}(\pp) \le \depth_{R_\pp} N_{\pp} = 1$.

(3)
It follows from an argument similar to that in (2).
\end{proof}

We explain the relationship between $f_{\catX}$ and the higher associated primes introduced in Section \ref{s:KE-closed}.
\begin{lem}\label{lem:function and higher ass}
Let $\catX$ be a subcategory of $\catmod R$.
Then $f^{-1}_\catX(\{0,1,\dots,n\})=A^n(\catX)$ for any $n\ge 0$.
\end{lem}
\begin{proof}
It follows from
\[
A^n(\catX)=\{\pp \mid \inf_{X\in \catX} \depth_{R_\pp}X_\pp \le n\}=\{\pp \mid f_\catX(\pp)\le n\}=f^{-1}_\catX(\{0,1,\dots,n\}).
\]
\end{proof}

Thus, we obtain the following from the main theorem of Section \ref{s:KE-closed} (\cref{thm:KE-closed reconstruction}).
\begin{prp}\label{prp:X=XfX}
For any KE-closed subcategory $\catX$ of $\catmod R$,
we have $\catX=\catX_{f_\catX}$.
\end{prp}
\begin{proof}
The inclusion $\catX\subseteq \catX_{f_\catX}$ is clear.
Suppose $M\in \catX_{f_\catX}$.
Then $A^0(M)\subseteq f^{-1}_\catX(\{0\})=A^0(\catX)$ and $A^1(M)\subseteq f^{-1}_\catX(\{0,1\})=A^1(\catX)$.
Therefore, \cref{thm:KE-closed reconstruction} ensures that $M\in \catX$.
We then conclude that $\catX=\catX_{f_\catX}$.
\end{proof}

Next, we characterize a function $\Spec R \to \bbN\cup\{\infty\}$ such that $f=f_{\catX_f}$.
Consider the following condition for a function $f\colon \Spec R \to \bbN \cup\{\infty\}$,
which is a natural modification of an element of $\bbF(\catmod R)$ in \cite{Takahashi2} for our setting:
\begin{itemize}
\item[\textbf{(F)}] For any $\pp\in \Spec R$, there exists $E\in\catmod R$ such that $f(\pp)=\depth_{R_{\pp}} E_\pp$ and $f(\qq)\le \depth_{R_{\qq}} E_{\qq}$ for any $\qq\in \Spec R$.
\end{itemize}

\begin{prp}\label{prp:fuctions F}
The following are equivalent for a function $f\colon \Spec R \to \bbN\cup\{\infty\}$.
\begin{enur}
\item
$f=f_{\catX_f}$ holds.
\item
$f=f_{\catX}$ for some subcategory $\catX$ of $\catmod R$.
\item
$f$ satisfies the condition \textrm{\textbf{(F)}}.
\end{enur}
\end{prp}
\begin{proof}
(i)$\imply$(ii):
It is clear.

(ii)$\imply$(iii):
Let $\catX$ be a subcategory of $\catmod R$.
We prove $f_{\catX}$ satisfies the condition \textrm{\textbf{(F)}}.
Take any $\pp \in \Spec R$.
If $f_{\catX}(\pp)=\infty$, then $E=0$ is a desired module.
If $f_{\catX}(\pp)<\infty$, then there is $X \in \catX$ such that $f_{\catX}(\pp) = \depth_{R_{\pp}}X_\pp$.
We also have $f_{\catX}(\qq) \le \depth_{R_{\qq}}X_\qq$ for any $\qq \in \Spec R$
since $X \in \catX$.
This means $E:=X$ is a desired module.

(iii)$\imply$(i):
Suppose that $f$ satisfies the condition \textrm{\textbf{(F)}}.
Then for any $\pp \in \Spec R$, there exists $E\in\catmod R$ such that $f(\pp)=\depth_{R_\pp} E_\pp$ and $f(\qq)\le \depth_{R_{\qq}} E_{\qq}$ for any $\qq\in \Spec R$.
We have $E\in \catX_f$, and hence $f_{\catX_f}(\pp) \le \depth_{R_{\pp}} E_{\pp}=f(\pp)$.
Thus, we get $f_{\catX_f} \le f$.
Because we have already seen $f_{\catX_f} \ge f$,
we obtain $f=f_{\catX_f}$.
\end{proof}

We will give a sufficient condition for functions to satisfy the condition \textrm{\textbf{(F)}} without referring to information of modules (\cref{prp:2-Bass satisfies F}).
We study properties of $f_{\catX}$ for this.
\begin{lem}\label{prp:funct ass to cat}
Let $\catX$ be a subcategory of $\catmod R$.
\begin{enua}
\item 
We have $\dom(f_\catX)=\Supp\catX$ and $f_\catX^{-1}(0)=\Ass\catX$
(see \eqref{eq:dom} for the definition of $\dom(f)$).
\item
The subset $\dom(f_\catX) \subseteq \Spec R$ is specialization-closed,
that is, for any inclusion $\pp\subseteq \qq$ in $\Spec R$, 
we have that $\pp\in\dom(f_\catX)$ implies $\qq\in\dom(f_\catX)$
\item
For any minimal element $\pp \in \Supp \catX$ with respect to inclusions,
we have $f_\catX(\pp)=0$.
\item
For any inclusion $\pp \subseteq \qq$ in $\Supp\catX$,
we have $f_\catX(\qq) \le f_\catX(\pp) + \htt\qq/\pp$.
\end{enua}
\end{lem}
\begin{proof}
(1)
We obtain the following from \cref{lem:function and higher ass}:
\[
\dom(f_\catX)=\bigcup_{n\ge 0}f^{-1}_\catX(\{0,1,\dots,n\})=\bigcup_{n\ge 0}A^n(\catX)=\Supp \catX,\quad
f^{-1}(\{0\})=A^0(\catX)=\Ass \catX.
\]

(2)
It follows from (1) since $\Supp\catX$ is specialization-closed.

(3)
We can easily see that every minimal element in $\Supp \catX$ belongs to $\Ass \catX$.
Thus, the claim follows from (1).

(4)
Take $M\in \catX$ such that $f_{\catX}(\pp)= \depth M_{\pp}$.
Then $f_{\catX}(\qq)=\inf_{X\in \catX} \left\{\depth X_{\qq}\right\} \le \depth M_{\qq}$ holds.
We have the following inequalities by \cref{Bass lemma}:
\[
f_{\catX}(\qq) \le \depth_{R_\qq}M_\qq \le \depth M_{\pp} +\htt \qq/\pp
= f_{\catX}(\pp)+\htt \qq/\pp.
\]
\end{proof}

Motivated by the above lemma,
we introduce the following class of functions.
A proper inclusion $\pp \subsetneq \qq$ of prime ideals is \emph{saturated}
if there is no prime ideal $\mathfrak{r}$ such that $\pp \subsetneq \mathfrak{r} \subsetneq \qq$.
\begin{dfn}\label{dfn:Bass fct}
A function $f \colon \Spec R \to \bbN \cup \{\infty\}$ is called a \emph{Bass function}
if it satisfies the following three conditions:
\begin{description}
\item[(B1: support condition)]
The subset $\dom(f)$ is specialization-closed in $\Spec R$.
\item[(B2: minimal prime condition)]
For any minimal element $\pp$ of $\dom(f)$,
we have $f(\pp)=0$.
\item[(B3: Bass condition)]
For any saturated inclusion $\pp \subsetneq \qq$ in $\dom(f)$,
we have $f(\qq) \le f(\pp) + 1$.
\end{description}
Moreover, for $n\ge 0$ we call it an \emph{$n$-Bass function} if it satisfies
$f(\pp) \le n$ for any $\pp \in \dom(f)$.
We denote by $\Bass_n(\Spec R)$ the set of $n$-Bass functions on $\Spec R$.
\end{dfn}

\begin{ex}\label{ex:Bass fct}
We give examples of Bass functions.
\begin{enua}
\item
For a subcategory $\catX$ of $\catmod R$,
the function $f_{\catX}$ associated to $\catX$ is a Bass function by \cref{prp:funct ass to cat}.
\item
For a module $M \in \catmod R$, a function defined $\pp \mapsto \depth_{R_\pp}M_\pp$ is a Bass function.
Indeed, this function is nothing but the function $f_{\{M\}}$ associated to the subcategory consisting of the single object $M$.
\item
The height function $\pp \mapsto \htt \pp$ is a Bass function
if $R$ is catenary and locally equidimensional (e.g., if $R$ has a maximal Cohen-Macaulay $R$-module with full support).
\item
The function $f_{\Phi}$ associated to a subset $\Phi$ of $\Spec R$ (see \cref{ex:X_f}) is a $1$-Bass function.
Moreover, we can easily see that every $1$-Bass function is of this form.
\end{enua}
\end{ex}

\begin{rmk}\hfill
\begin{enua}
\item
We can define Bass functions on any poset.
Thus, Bass functions only depend on the poset structure of $\Spec R$. 
\item
If a Bass function $f$ satisfies $f(\pp) \le \depth R_{\pp}$ for any $\pp \in \Spec R$,
then a function $\Spec R \to \bbN$ defined by $\pp \mapsto \depth R_{\pp} - f(\pp)$ is a \emph{moderate function},
which is introduced in \cite{Takahashi2} to classify dominant resolving subcategories.
\end{enua}
\end{rmk}

Let us give some basic properties of Bass functions.
The following lemma will be used frequently.
For a subset $\Phi$ of $\Spec R$ and its element $\pp$,
define the \emph{relative height} of $\pp$ in $\Phi$ by
$\htt_{\Phi} \pp := \sup\{\htt \pp/\qq \mid \qq \subseteq \pp,\; \qq\in \Phi\}$.
\begin{lem}\label{Bass fct basic}
Let $f \colon \Spec R \to \bbN \cup \{\infty\}$ be a Bass function.
\begin{enua}
\item
For any inclusion $\pp \subseteq \qq$ in $\dom(f)$,
we have $f(\qq) \le f(\pp) + \htt\qq/\pp$.
\item
For any $\pp \in \dom(f)$,
we have $f(\pp) \le \htt_{\dom}\pp$, where we put $\htt_{\dom}:=\htt_{\dom(f)}$.
In particular, we have $f(\pp) \le \htt\pp$.
\end{enua}
\end{lem}
\begin{proof}
(1) Take a saturated chain $\pp=\pp_0 \subsetneq \pp_1 \subsetneq \cdots \subsetneq \pp_n=\qq$ of prime ideals with $n=\htt \qq/\pp$.
Every $\pp_i$ belongs to $\dom(f)$ by \textbf{(B1)}.
Applying \textbf{(B3)} repeatedly, we obtain that
\[
f(\qq) \le f(\pp_{n-1})+1 \le \cdots \le f(\pp_1)+n-1 \le f(\pp_0)+n=f(\pp)+\htt(\qq/\pp).
\]

(2) Take $\pp_0 \in \dom(f)$ with $\htt_{\dom}\pp = \htt \pp/\pp_0$.	
Then $\pp_0$ is a minimal element of $\dom(f)$,	
and hence $f(\pp_0)=0$ by \textbf{(B2)}.
From (1), we have $f(\pp) \le f(\pp_0) + \htt \pp/\pp_0 =\htt_{\dom}\pp$.
\end{proof}

The functions $f_\catX$ associated to subcategories are Bass functions as mentioned in \cref{ex:Bass fct}.
Conversely, Bass functions sometimes come from subcategories.
We need the following lemma to prove this.
Recall that $M\in \catmod R$ \emph{satisfies Serre's condition $(S_2)$}, or simply, $M$ is \emph{$(S_2)$}	
if $\depth_{R_\pp} M_\pp \ge \min\{2,\htt\pp\}$ holds for any $\pp \in \Spec R$.
We denote by $\Se_2(R)$ the subcategory consisting of $M\in\catmod R$ satisfying $(S_2)$.
\begin{lem}\label{S2 deform}
Let $M\in \catmod R$ satisfying $(S_2)$ and $\pp \in \Supp M$.
\begin{enua}
\item
If $\htt \pp \ge 1$, then there is a submodule $N$ of $M$ such that
\[
\depth_{R_\qq} N_{\qq}=
\begin{cases}
1 & \text{if $\qq= \pp$,} \\
\ge \min\{2, \htt \qq\} & \text{if $\qq\ne \pp$.}
\end{cases}
\]
\item
If $\htt\pp \ge 2$,
then there exists $K \in \catmod R$ such that
\[
\depth_{R_\qq} K_{\qq}=
\begin{cases}
2 & \text{if $\qq=\pp$,} \\
\ge \min\{2, \htt \qq\} & \text{if $\qq \ne \pp$.}
\end{cases}
\]
\end{enua}
\end{lem}
\begin{proof}
It follows from \cref{prp:depth 1 2}.
\end{proof}
We now prove that $2$-Bass functions satisfy the condition \textbf{(F)}	
under the existence of nonzero $(S_2)$-modules.	
However, in the next subsection, we will see that this assumption is satisfied for a large class of commutative noetherian rings, including homomorphic images of Cohen-Macaulay rings and excellent rings.

\begin{prp}\label{prp:2-Bass satisfies F}
Assume that $\Se_2(R/\pp) \ne 0$ for any $\pp \in \Spec R$.
Then for any $2$-Bass function $f\colon \Spec R \to \bbN \cup\{\infty\}$,
there is some subcategory $\catX$ of $\catmod R$ such that $f=f_{\catX}$.
\end{prp}
\begin{proof}
We prove that a $2$-Bass function $f$ satisfies the condition \textbf{(F)}.
Take $\pp\in \Spec R$.
We want to find $E\in \catmod R$ such that $\depth_{R_\pp}E_\pp=f(\pp)$ and $\depth_{R_\qq} E_\qq \ge f(\qq)$ for any $\qq \in \Spec R$.
Note that for any $0 \ne M \in \Se_2(R/\pp)$,
we have $\Supp_R M = V(\pp)$ and $\depth_{R_\qq} M_\qq \ge \min\{2,\htt\qq/\pp\}$ for any $\qq \in V(\pp)$.
\begin{description}
\item[(The case $f(\pp)=\infty$)]
The zero module $E=0$ is a desired one.
\item[(The case $f(\pp)=0$)]
There exists $0 \ne X \in \Se_2(R/\pp)$.
Then $E=X$ is a desired one.
Indeed, we have $f(\pp)= 0 =\depth_{R_\pp} X_\pp$.
If $\qq \not\in V(\pp)$, then $f(\qq)\le \infty = \depth_{R_\qq} X_\qq$.
If $\qq \supsetneq \pp$, 
we have $f(\qq)\le f(\pp)+\htt\qq/\pp=\htt\qq/\pp$ by \cref{Bass fct basic} (1).
Since $f(\qq)\le 2$ by the assumption, we have
\[
f(\qq)\le \min\{2,\htt\qq/\pp\} \le \depth_{R_\pp} X_\pp.
\]
\end{description}
Take $\pp_0 \in \dom(f)$ such that $\htt_{\dom}\pp = \htt\pp/\pp_0$,
and take $0\ne M \in \Se_2(R/\pp_0)$.
Then $f(\pp_0)=0$ by \textbf{(B2)}.
If $\qq \in V(\pp_0)$,
we have $f(\qq)\le f(\pp_0)+\htt\qq/\pp_0=\htt\qq/\pp_0$ for any $\qq \in V(\pp_0)$
by \cref{Bass fct basic} (1).
Thus, we have $f(\qq)\le \min\{2,\htt\qq/\pp_0\}$ for any $\qq \in V(\pp_0)$ by the assumption.
\begin{description}
\item[(The case $f(\pp)=1$)]
In this case, we have the following by \cref{Bass fct basic} (2):
\[
\depth_{R_\pp} M_\pp \ge \min\{2,\htt\pp/\pp_0\} = \min\{2,\htt_{\dom}\pp\} \ge \min\{2,f(\pp)\} \ge 1.
\]
Thus, by \cref{S2 deform} (1), there is a submodule $N$ of $M$ such that
\[
\depth_{R_\qq} N_{\qq}=
\begin{cases}
\infty & \text{if $\qq\not\in V(\pp_0)$,} \\
1 & \text{if $\qq = \pp$,} \\
\ge \min\{2, \htt \qq/\pp_0\} & \text{if $\qq \in V(\pp_0)\setminus \{\pp\}$.}
\end{cases}
\]
Then $E=N$ is a desired one.
Indeed, we have $f(\pp)=1 = \depth_{R_\pp} N_\pp$.
If $\qq \not\in V(\pp_0)$, then $f(\qq)\le \infty = \depth_{R_\qq}N_\qq$.
If $\qq \in V(\pp_0) \setminus \{\pp\}$,
we have $f(\qq) \le \min\{2,\htt\qq/\pp_0\} \le \depth_{R_\qq}N_\qq$.
\item[(The case $f(\pp)=2$)]
In this case, we have $\depth_{(R/\pp_0)_\pp} M_\pp \ge 2$.
Thus, by \cref{S2 deform} (2), there is $K\in \catmod R$ such that
\[
\depth_{R_\qq} K_{\qq}=
\begin{cases}
\infty & \text{if $\qq\not\in V(\pp_0)$,} \\
2 & \text{if $\qq = \pp$,} \\
\ge \min\{2, \htt \qq/\pp_0\} & \text{if $\qq \in V(\pp_0)\setminus \{\pp\}$.}
\end{cases}
\]
Then $E=K$ is a desired one.
\end{description}
This finishes the proof.
\end{proof}

Finally, we also discuss the case of 1-Bass functions.
\begin{prp}\label{prp:1-Bass satisfies F}
For any $1$-Bass function $f\colon \Spec R \to \bbN \cup\{\infty\}$,
there is some subcategory $\catX$ of $\catmod R$ such that $f=f_{\catX}$.
\end{prp}
\begin{proof}
Note that the nonzero $R/\pp$-module $R/\pp$ satisfies ($S_1$).
We can prove that $1$-Bass functions satisfy the condition \textbf{(F)} by an argument similar to that in the proof of \cref{prp:2-Bass satisfies F}.
\end{proof}
The results of this subsection are summarized as follows:

\begin{equation}\label{diag:summary KE}
\begin{tikzcd}
\{\text{subcategories of $\catmod R$}\} \ar[r,"{f_{(-)}}",yshift=2.5pt]  & \{\text{functions $\Spec R \to \bbN\cup\{\infty\}$}\} \ar[l,"\catX_{(-)}",yshift=-2.5pt] \\
\ke(\catmod R) \ar[r,yshift=2.5pt,hook] \ar[u,phantom,"\subseteq"sloped] & \Bass_2(\Spec R) \ar[l,yshift=-2.5pt,"{\text{(i)}}",two heads] \ar[u,phantom,"\subseteq"sloped] \\
\torf(\catmod R) \ar[r,"\simeq"] \ar[u,phantom,"\subseteq"sloped] & \Bass_1(\Spec R) \ar[l,"{\text{(ii)}}"] \ar[u,phantom,"\subseteq"sloped] \\
\serre(\catmod R) \ar[r,"\simeq"] \ar[u,phantom,"\subseteq"sloped] & \Bass_0(\Spec R) \ar[l,"{\text{(iii)}}"] \ar[u,phantom,"\subseteq"sloped] \\
\end{tikzcd}
\end{equation}

\begin{itemize}
\item
The maps in \eqref{eq:subcat funct corresp} restrict to the maps (i)--(iii) by \cref{prp:catX_f KE,prp:bound KE}.
\item
In (i), $\catX_{f_\catX}= \catX$ holds for any $\catX \in \ke(\catmod R)$ by \cref{prp:X=XfX}.
If $\Se_2(R/\pp)\ne0$ for any $\pp \in \Spec R$ (e.g., $R$ is $(S_2)$-excellent, see the next subsection), then $f_{\catX_f}=f$ for any $f\in \Bass_2(\Spec R)$ by \cref{prp:fuctions F,prp:2-Bass satisfies F}.
\item
The maps in (ii) and (iii) are bijections by \cref{prp:X=XfX,prp:fuctions F,prp:1-Bass satisfies F}
\end{itemize}

\subsection{The existence of modules satisfying Serre's condition {$(S_2)$}}\label{ss:S2}
In this subsection, we discuss the existence of nonzero modules satisfying $(S_2)$.
We first recall the definition.	
\begin{dfn}\label{def:S_n}	
Let $M\in \catmod R$ and $n$ be a nonnegative integer.	
\begin{enua}
\item	
$M$ \emph{satisfies Serre's condition $(S_n)$}, or simply, $M$ is \emph{$(S_n)$}	
if $\depth_{R_{\pp}}M_{\pp} \ge \min\{n, \htt\pp\}$ holds for any $\pp \in \Spec R$.	
\item	
A commutative noetherian ring $R$ is \emph{$(S_n)$}	
if it is $(S_n)$ as an $R$-module.	
\item	
The \emph{$(S_n)$-locus} of $R$ is defined as follows:	
\[	
U_{(S_n)}:=U_{(S_n)}(R):=\{\pp\in \Spec R \mid \text{$R_\pp$ is $(S_n)$}\}.	
\]	
\end{enua}	
\end{dfn}	
﻿	
For the existence of nonzero $(S_2)$-modules,	
we need the notion of $(S_2)$-excellence, which is introduced in \cite{Ces}
(cf.\ \cite{Takahashi3} for rings).
\begin{dfn}	
For a positive integer $n$,	
we say that $R$ is \emph{$(S_n)$-quasi-excellent}	
if it satisfies the following conditions:	
\begin{enua}	
\item	
The formal fibers of $R_\pp$ are $(S_n)$ for any $\pp\in\Spec R$.	
That is, any fiber of the completion $R_\pp \to \widehat{R_\pp}$ is $(S_n)$.	
\item	
The $(S_n)$-locus of every finitely generated $R$-algebra is open.	
\end{enua}	
In addition, if it is universally catenary, we say that $R$ is \emph{$(S_n)$-excellent}.	
\end{dfn}	
﻿	
\begin{rmk}\label{S_n-excellent}\hfill	
\begin{enua}	
\item	
Our definition of Serre's condition $(S_n)$ for modules differs from that of \cite{Ces}.	
Compare \cref{def:S_n} with \cite[1.14]{Ces}.	
However, both notions coincide for rings (or schemes).	
\item	
The ring $R$ is $(S_n)$-(quasi-)excellent if so is $\Spec R$ as a scheme in the sense of \cite[2.10]{Ces}.	
It follows from a discussion similar to that in \cite[Remark2.8 (1)]{Takahashi3}.	
\item	
If $R$ is $(S_n)$-(quasi-)excellent, then so is every finitely generated $R$-algebra and its localization (see \cite[2.10]{Ces}).	
\item	
If $R$ is $(S_n)$-(quasi-)excellent for every $n$,	
it is said to be \emph{CM-(quasi-)excellent}.	
\item	
The class of CM-excellent rings includes excellent rings, a homomorphic image of a Cohen-Macaulay ring, commutative noetherian rings with a dualizing complex, and commutative noetherian rings having a maximal Cohen-Macaulay module with full support
(see \cite[Example 1.4 and Remark 1.5]{Ces} and \cite[Remark 2.8]{Takahashi3}).	
In particular, every Dedekind domain is CM-excellent, while there exists a non-excellent discrete valuation ring.	
\end{enua}	
\end{rmk}	
﻿	
\begin{prp}\label{S_2-ification}	
Every noetherian $(S_2)$-quasi-excellent ring $R$ has an \emph{$(S_2)$-ification}:	
there is a finite homomorphism $\phi\colon R \to S$ such that $S$ is $(S_2)$, locally equidimensional, and $(S_2)$-excellent.	
Moreover, if $R$ is an integral domain, then so is $S$, and $\phi$ is an injective map that induces an isomorphism $Q(R) \isoto Q(S)$.	
\end{prp}	
\begin{proof}	
It is a direct translation of \cite[Corollary 2.14]{Ces} for commutative rings.	
We write it down for the reader's convenience.	
By \cite[Corollary 2.14]{Ces}, we have a finite morphism $f\colon X \to \Spec R$	
that is an isomorphism over the $(S_2)$-locus $U_{(S_2)}(R)$	
such that $X$ is $(S_2)$, locally equidimensional, $(S_2)$-excellent,	
and has $f^{-1}(U_{(S_2)}(R))$ as a dense open subset.	
Since $f$ is finite, the scheme $X$ is affine (cf.\ \cite[Proposition and Definition 12.9]{GW-AG}),	
and put $X=\Spec S$.	
Then the corresponding morphism $\phi \colon R \to S$ satisfies the desired properties in the first part.	
﻿	
To prove the remaining part, we suppose that $R$ is an integral domain.	
Since $U_{(S_2)}(R)$ and $f^{-1}(U_{(S_2)}(R))$ are dense open in the noetherian schemes,	
we have	
\[	
\Min(R) = \Min(U_{(S_2)}(R)) \iso \Min f^{-1}(U_{(S_2)}(R)) = \Min(S),	
\]	
where, for a scheme $Z$, $\Min Z$ denotes the set of generic points of $Z$.	
This means that $\Spec S$ is irreducible.	
Considering the stalks of the generic points,	
the isomorphism $f^{-1}(U_{(S_2)}(R)) \isoto U_{(S_2)}(R)$	
induces an isomorphism $Q(R) \iso Q(S)$ since $S$ has no embedded points.	
Thus, $S$ is reduced, and hence it is an integral domain.	
Then $\phi$ is injective by the following:	
\[	
\begin{tikzcd}[ampersand replacement=\&]	
R \ar[d,hook] \ar[r,"\phi"] \& S \ar[d,hook]\\	
Q(R) \ar[r,"\simeq"] \& Q(S).	
\end{tikzcd}	
\]	
\end{proof}	
﻿	
\begin{lem}\label{S_n descend univ catenary}	
Let $R \inj S$ be a finite ring extension of noetherian integral domains.	
Suppose that $R$ is universally catenary.	
If $M \in \catmod S$ is $(S_n)$,	
then it is also $(S_n)$ as an $R$-module.	
\end{lem}	
\begin{proof}	
Take any $\pp \in \Spec R$.	
There exists $\qq \in \Spec S$ lying over $\pp$	
such that $\depth_{R_\pp} M_\pp = \depth_{S_\qq} M_\qq$ by \cref{comparison of depth}.	
Since $R \inj S$ is a finite homomorphism,	
we have $\htt\pp= \htt \qq$ by the dimension formula (cf.\ \cite[Theorem 15.5, 15.6]{CRT}).	
Thus, we obtain	
\[	
\depth_{R_\pp} M_\pp = \depth_{S_\qq} M_\qq \ge \min\{n,\htt\qq\} = \min\{n,\htt\pp\}.	
\]	
This proves $M$ is $(S_n)$ as an $R$-module.	
\end{proof}	

Now we can prove the existence of nonzero $(S_2)$-modules over a $(S_2)$-excellent domain.
\begin{cor}\label{exist S_2}
If $R$ is $(S_2)$-excellent, then $\Se_2(R/\pp) \ne 0$ for all $\pp\in \Spec R$.	
\end{cor}	
\begin{proof}	
Note that $R/\pp$ is $(S_2)$-excellent for any $\pp \in \Spec R$ by \cref{S_n-excellent} (3).	
Take a $(S_2)$-ification $R/\pp \inj S$ (\cref{S_2-ification}).	
Then $0\ne S_{R/\pp} \in \Se_2(R/\pp)$ by \cref{S_n descend univ catenary}.	
\end{proof}

\begin{cor}\label{cor:bij KE 2-Bass}
If $R$ is $(S_2)$-excellent,
then there is an order-reversing bijection between $\ke(\catmod R)$ and $\Bass_2(\Spec R)$.
\end{cor}
\begin{proof}
It follows from \cref{prp:X=XfX,prp:2-Bass satisfies F,exist S_2}.
\end{proof}


From the existence of nonzero $(S_2)$-modules,
we can also prove the following proposition that extends the main theorem of \cite{KS} (cf.\ \cref{KE=torf in 1-dim}).
\begin{prp}\label{prp:KE=torf for S_2 excellent}
Assume $R$ is $(S_2)$-excellent.
Then the equality $\ke(\catmod R)=\torf(\catmod R)$ holds true if and only if $\dim R\le 1$.
\end{prp}
\begin{proof}
From \cref{prp:KE=torf via depth}, it suffices to show that
$\dim R\le 1$ if and only if $\depth_{R_\pp}M_\pp\le 1$ for any $M\in \catmod R$ and $\pp \in \Supp M$.
Since the sufficiency is clear, we prove the necessity.
Suppose that $\dim R \ge 2$.
Take $\pp \in \Spec R$ such that $\dim R/\pp \ge 2$.
There exists $0 \ne M \in \Se_2(R/\pp)$ by \cref{exist S_2}.
Then, for any $\qq \in V(\pp)$ such that $\htt(\qq/\pp) \le 2$,
we have $M_\qq\ne 0$ and $\depth_{R_\qq} M_\qq \ge 2$.
\end{proof}

\subsection{Examples and consequences}\label{ss:ex Bass}
In this subsection, we give some consequences and examples of our classification.

For $n \ge 0$,
we denote by $\Fct_n(\Spec R)$ be the set of functions $f\colon \Spec R \to \bbN\cup\{\infty\}$ such that $f(\pp) \le n$ for any $\pp \in \dom(f)$.
Thus, $\Bass_n(\Spec R)=\Bass(\Spec R)\cap \Fct_n(\Spec R)$.
The following observation is useful for enumerating all KE-closed subcategories, possibly with repetitions.
\begin{lem}\label{prp:enumerate KE}
Let $R$ be a commutative noetherian ring (not necessarily $(S_2)$-excellent).
For any subset $\Psi$ of $\Fct_2(\Spec R)$ containing $\Bass_2(\Spec R)$,
we have $\ke(\catmod R)=\{\catX_f\mid f\in \Psi\}$.
\end{lem}
\begin{proof}
It easily follows from \cref{prp:catX_f KE,prp:bound KE,prp:X=XfX}.
\end{proof}

\begin{ex}
Suppose that $R$ is a domain.
Let $(0) \subsetneq \pp_1 \subsetneq \pp_2$ be a saturated chain of $\Spec R$.
Then the possible values of Bass functions on $\Spec R$ can be illustrated as follows:
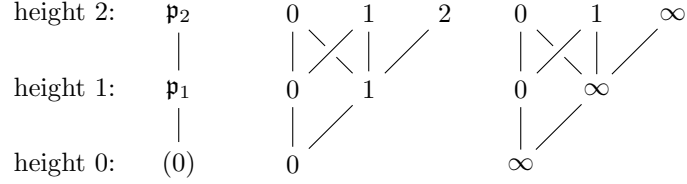
\[
\begin{tikzpicture}
\node (ht0) at (0,0) {height 0:};
\node [above of=ht0] (ht1) {height 1:};
\node [above of=ht1] (ht2) {height 2:};

\node (generic) at (1.5,0) {$(0)$};
\node [above of=generic] (p) {$\pp_1$};
\node [above of=p] (m) {$\pp_2$};

\draw (generic)--(p);
\draw (p)--(m);

\node (0) at (3,0) {$0$};
\node [above of=0] (p0) {$0$};
\node [above of=p0] (m0) {$0$};
\node [right of=p0] (p1) {$1$};
\node [right of=m0] (m1) {$1$};
\node [right of=m1] (m2) {$2$};

\draw (0)--(p0);
\draw (0)--(p1);
\draw (p0)--(m0);
\draw (p1)--(m0);
\draw[white,line width=6pt] (p0) --(m1);
\draw (p0)--(m1);
\draw (p1)--(m1);
\draw (p1)--(m2);

\node (infty) at (6,0) {$\infty$};
\node [above of=infty] (q0) {$0$};
\node [right of=q0] (qinf) {$\infty$};
\node [above of=q0] (n0) {$0$};
\node [right of=n0] (n1) {$1$};
\node [right of=n1] (ninf) {$\infty$};

\draw (infty)--(q0);
\draw (infty)--(qinf);
\draw (q0)--(n0);
\draw (qinf)--(n0);
\draw[white,line width=6pt] (q0) --(n1);
\draw (q0)--(n1);
\draw (qinf)--(n1);
\draw (qinf)--(ninf);
\end{tikzpicture}
\]
This shows, for example, that if $f(\pp_1)=\infty$, then the possible values of $f(\pp_2)$ are $0$, $1$, or $\infty$.
Thus, if $(R,\mm)$ is a two-dimensional local domain,
there is at most one $2$-Bass function that is not $1$-Bass,
namely, the height function $\htt \colon \pp \mapsto \htt \pp$.
It is not a Bass function in general,
but it becomes one when $R$ is catenary (cf.\ \cref{ex:Bass fct}).
As $\catX_{\htt}$ coincides with $\cm R$,
we have $\ke(\catmod R) = \torf(\catmod R)\cup\{ \cm R \}$ by \cref{prp:enumerate KE}.
Note that $\cm R$ can be zero, in which case $\cm R \in \torf(\catmod R)$.
\end{ex}

We now classify the Bass functions having maximal values.
We denote by $\Assh R$ the set of prime ideals $\pp$ such that $\dim R/\pp = \dim R$.
Also, recall that $\Phi^{\mathsf{up}}:=\{\pp \in \Spec R \mid \exists \qq \in \Phi,\, \qq \subseteq \pp \}$ for a subset $\Phi$ of $\Spec R$.
\begin{lem}\label{prp:classify Bass with maximal values}
Suppose that $(R,\mm)$ is a $d$-dimensional local ring.
Consider the following two sets:
\begin{enua}
\item
the set of $d$-Bass functions on $\Spec R$ that are not $(d-1)$-Bass.
\item
the set of non-empty subsets of $\Assh R$.
\end{enua}
Then the assignment $f \mapsto \dom(f) \cap \Assh(R)$ defines an injection from (1) to (2).
Moreover, it is bijective when $R$ is catenary.
In this case, the inverse map is given by $\Phi \mapsto g_\Phi$,
where
\[
g_\Phi(\pp):=
\begin{cases}
\htt(\pp) & \text{if $\pp \in \Phi^{\up}$},\\
\infty & \text{if $\pp \not\in \Phi^{\up}$}.
\end{cases}
\]
\end{lem}
\begin{proof}
Take a $d$-Bass function $f$ that is not $(d-1)$-Bass.
Then $f(\mm)=d$ holds by definition.
As $d=f(\mm)\le \htt_{\dom} \mm \le \htt\mm=d$,
there is $\pp_0 \in \dom(f)$ such that $\htt\mm/\pp_0=d$.
Thus, the subset $\dom(f)\cap \Assh(R)$ is non-empty.
We also have $f(\pp)=\htt\pp$ for any $\pp \in \dom(f)$ by the following inequalities:
\[
d=f(\mm)\le f(\pp)+\htt(\mm/\pp) \le \htt\pp+\htt(\mm/\pp) \le \htt\mm=d.
\]
Therefore, we have $f=g_{\dom(f)\cap\Assh(R)}$,
which proves the injectivity.

Suppose that $R$ is catenary.
To prove the bijectivity, it is enough to show that $g_\Phi$ is a Bass function for any non-empty subset $\Phi$ of $\Assh R$.
The only nontrivial condition to check is \textbf{(B3)}.
We first prove that $\htt\pp=d-\htt(\mm/\pp)$ for any $\pp \in (\Assh R)^{\up}$.
As $R$ is catenary,
for any inclusion $\qq \subseteq \pp$ in $\Spec R$,
we have $\htt(\pp/\qq)=\htt(\mm/\qq)-\htt(\mm/\pp)$.
Thus, we have the following equality for any $\pp \in (\Assh R)^{\up}$:
\[
\htt \pp 
= \max\{\htt(\pp/\qq) \mid \qq \subseteq \pp\}
= \max\{\htt(\mm/\qq) \mid \qq \subseteq \pp\} - \htt(\mm/\pp)
= d - \htt(\mm/\pp).
\]
Here, the fact $\pp \in (\Assh R)^{\up}$ is used in the last equality.
Take any saturated chain $\pp \subsetneq \qq$ in $\dom(g_{\Phi})=\Phi^{\up}$.
Then we obtain the following desired equality which proves $g_\Phi$ satisfies \textbf{(B3)}:
\[
g_{\Phi}(\qq)=\htt \qq = d - \htt(\mm/\qq) = d - (\htt(\mm/\pp)-1) = \htt\pp +1 = g_{\Phi}(\pp)+1.
\]
\end{proof}

This lemma immediately yields a classification of KE-closed subcategories in the two-dimensional local case.
\begin{cor}\label{prp:classify KE but not torf}
Suppose that $R$ is a two-dimensional local ring.
Consider the following two sets:
\begin{enua}
\item
the set of KE-closed subcategories of $\catmod R$ that are not torsion-free classes.
\item
the set of non-empty subsets of $\Assh R$.
\end{enua}
Then there is an injection from (1) to (2).
Moreover, this map is bijective if $R$ is $(S_2)$-excellent.
In this case,
the KE-closed subcategory $\catX$ corresponding to a non-empty subset $\Phi$ of $\Assh R$ is the following:
\[
\catX:=\{M\in \catmod R \mid \text{$\Supp M \subseteq \Phi^{\up}$ and $\depth_{R_{\pp}}M_{\pp} \ge \htt\pp$ for any $\pp \in \Phi^{\up}$}\}.
\]
\end{cor}
\begin{proof}
There is an injection $\ke(\catmod R)\setminus \torf(\catmod R) \inj \Bass_2(\Spec R) \setminus \Bass_1(\Spec R)$ (cf.\ \eqref{diag:summary KE}).
By \cref{prp:classify Bass with maximal values}, we also have an injection from $\Bass_2(\Spec R) \setminus \Bass_1(\Spec R)$ to the set of non-empty subsets of $\Assh R$.
Moreover, each of these maps is bijective when $R$ is $(S_2)$-excellent.
\end{proof}

From this, we obtain the following interesting observation.
\begin{cor}
If $R$ is a two-dimensional local ring,
there are only finitely many KE-closed subcategories that are not torsion-free classes. \qed
\end{cor}

\begin{ex}
Let $\bbk$ be a field, and let $R:=\bbk[[x,y,z]]/(xy,xz)$.
Then $R$ is an $(S_2)$-excellent two-dimensional local ring,
which is not equidimensional and $\Assh R = \{(x)\}$.
By \cref{prp:classify KE but not torf},
there is exactly one KE-closed subcategory $\catX$ given by
\[
\catX=\{M\in \catmod R \mid \text{$\Supp M \subseteq V(x)$ and $\depth_{R_{\pp}}M_{\pp} \ge \htt\pp$ for any $\pp \in V(x)$}\}.
\]
\end{ex}

Next, we describe the classification of KE-closed subcategories in terms of higher associated primes (\cref{dfn:A^n}).
It reveals the relationships between our classification of KE-closed subcategories and those of other classes of subcategories.
For a subset $\Phi$ of $\Spec R$,
we denote by 
\[
\Phi^{\cov}:= \{\pp\in \Spec R  \mid \exists \qq \in \Phi,\, \qq \subsetneq \pp\colon\text{saturated}\}.
\]


\begin{dfn} \label{dfn:Bass sequence}
A sequence $(\Phi_i)_{i\in\mathbb{N}}$ of subsets of $\Spec R$ is called a \emph{Bass sequence} if it satisfies the following conditions:
\begin{enur}
\item for any $i\ge 0$, $\Phi_i \cup \Phi_i^{\cov} \subseteq \Phi_{i+1}$, and
\item $\Phi_0^\up=\bigcup_{i\ge 0}\Phi_i$.
\end{enur}
Moreover, if $\Phi_i=\Phi_0^\up$ for all $i\ge n$, then it is called an \emph{$n$-Bass sequence}.
\end{dfn}
As an $n$-Bass sequence $(\Phi_i)_{i\in\mathbb{N}}$ is determined by $(\Phi_i)_{i=0}^{n-1}$ for $n \ge 2$, and by $\Phi_0$ for $n=0,1$,
we often identify it with $(\Phi_i)_{i=0}^{n-1}$ in the former case, and with $\Phi_0$ in the latter.
Then $0$-Bass sequences are identified with specialization-closed subsets of $\Spec R$.
Also, $1$-Bass sequences are identified with subsets of $\Spec R$.
A $2$-Bass sequence is identified with a pair $(\Phi,\Psi)$ of subsets of $\Spec R$ such that $\Phi \cup \Phi^{\cov}\subseteq \Psi \subseteq \Phi^\up$.
We define an order on Bass sequences by
\[
(\Phi_i)_{i\in \bbN} \ge (\Psi_i)_{i\in \bbN} \quad :\Longleftrightarrow \quad
\text{$\Phi_i \supseteq \Psi_i$ for any $i \in \bbN$}.
\]
Bass sequences are naturally related to Bass functions.
\begin{lem}
There exist order-reversing bijections between the following two posets:
\begin{enua}
\item the set of Bass functions $f$.
\item the set of Bass sequences $\Phi=(\Phi_i)_{i\in \bbN}$.
\end{enua}
The bijections are given by $f \mapsto (f^{-1}\{0,1,\dots,i\})_{i\in \mathbb{N}}$ and $\Phi \mapsto [f_\Phi \colon \pp \mapsto \inf\{i \mid \pp\in \Phi_i\}]$.
Moreover, it restricts to a bijection between $n$-Bass functions and $n$-Bass sequences for each $n\ge 0$.
\end{lem}
\begin{proof}
We omit the proof since it is straightforward.
\end{proof}
The Bass function $f_{\catX}$ associated to a subcategory $\catX$ of $\catmod R$
corresponds to the Bass sequence $(A^i(\catX))_{i \in \bbN}$ via the bijection described above (cf.\ \cref{lem:function and higher ass}).
Thus, we can describe the classification of KE-closed subcategories in terms of higher associated primes and Bass sequences.
\begin{cor}\label{prp:classification via seq}
Assume $R$ is $(S_2)$-excellent.
There exist order-preserving bijections between the following two posets:
\begin{enua}
\item the set $\ke(\catmod R)$ of KE-closed subcategories $\catX$.
\item the set of $2$-Bass sequences $(\Phi,\Psi)$.
\end{enua}
The bijections are given by the assignments $\catX \mapsto (A^0(\catX),A^1(\catX))$ and $(\Phi,\Psi) \mapsto \catmod_\Phi^\ass R \cap \catmod_\Psi^1 R$.
Under this bijection, torsion-free classes $\catmod_\Phi^\ass R$ correspond to the pairs of the form $(\Phi,\Phi^\up)$.\qed
\end{cor}
Our classification is summarized as follows (cf.\ \eqref{diag:summary KE}):
\[
\begin{tikzcd}
\ke(\catmod R) \ar[r,"{f_{(-)}}",yshift=2.5pt,hook] & \Bass_2(\Spec R) \ar[l,"\catX_{(-)}",yshift=-2.5pt,two heads] \ar[r,"\simeq"] & \{\text{$2$-Bass sequences $(\Phi,\Psi)$ of $\Spec R$}\} \ar[l] \\
\torf(\catmod R) \ar[u,phantom,"\subseteq"sloped] \ar[r,"\simeq"] & \Bass_1(\Spec R) \ar[u,phantom,"\subseteq"sloped] \ar[r,"\simeq"] \ar[l] & \{\text{subsets of $\Spec R$}\} \ar[l] \ar[u,phantom,"\subseteq"sloped]\\
\serre(\catmod R) \ar[u,phantom,"\subseteq"sloped] \ar[r,"\simeq"] & \Bass_0(\Spec R) \ar[u,phantom,"\subseteq"sloped] \ar[r,"\simeq"] \ar[l] & \{\text{specialization-closed subsets of $\Spec R$}\} \ar[u,phantom,"\subseteq"sloped] \ar[l]
\end{tikzcd}
\]
Here, the middle bijection coincides with Takahashi's classification of torsion-free classes \cite{Takahashi}, and the bottom bijection coincides with Gabriel's classification of Serre subcategories \cite{Gab}.
Thus, our classification extends both the classical ones.

Bass functions have the advantage that their properties are easier to investigate, whereas Bass sequences have the advantage that examples can be constructed more easily.
Using the description of our classification in terms of Bass sequences, we also have the following remarkable observation.
\begin{prp}\label{prp:smallest KE}
Suppose that $R$ is $(S_2)$-excellent.
For any torsion-free class $\catF=\catmod_\Phi^\ass R$, 
the KE-closed subcategory corresponding to $(\Phi,\Phi\cup\Phi^{\cov})$ is the smallest KE-closed subcategory whose torsion-free closure is $\catF$.
\end{prp}
\begin{proof}
Put $\catX:=\catmod^0_{\Phi}R \cap \catmod^1_{\Phi\cup\Phi^{\cov}} R$.
As $A^0(\catX)=\Phi$ by \cref{prp:classification via seq}, its torsion-free closure is $\catmod_\Phi^\ass R$ (cf.\ \cref{prp:torf closure for comm ring}).
Take any KE-closed subcategory $\catY$ such that $\F(\catY)=\catF$.
Since $A^0(\catY)=A^0(\F(\catY))=\Phi$,
the $2$-Bass sequence associated to $\catY$ is $(\Phi, A^1(\catY))$.
As it is a Bass sequence, we obtain an inclusion $(\Phi, \Phi\cup \Phi^{\cov})\subseteq(\Phi, A^1(\catY))$ of $2$-Bass sequences.
This yields the inclusion $\catX \subseteq \catY$ of KE-closed subcategories by \cref{thm:KE-closed reconstruction},
and thus we obtain the desired conclusion.
\end{proof}

\begin{ex}
Suppose that $R$ is an $(S_2)$-excellent domain.
Then the category $\tf R$ of torsion-free $R$-modules 
is a torsion-free class with $\Ass(\tf R)=\{(0)\}$.
Thus, by \cref{prp:smallest KE},
the category $\Se_2(R) = \catmod^0_{\{(0)\}}R \cap \catmod^1_{\{\pp \mid \htt\pp \le 1\}} R$ of $R$-modules satisfying Serre's $(S_2)$-condition is
the smallest KE-closed subcategory whose torsion-free closure is $\tf R$.
\end{ex}


\end{document}